\documentclass[11pt,reqno]{amsart}
\usepackage[margin=2.55cm]{geometry}

\usepackage{amsmath, amsfonts, amssymb, amscd, amsthm, amsbsy, amscd}
\usepackage{bbm, epsf, calc, enumerate, graphicx, verbatim, setspace}
\usepackage{cite}
\usepackage[usenames,dvipsnames]{color}
\usepackage{mathrsfs, upgreek}
\usepackage{thmtools}
\usepackage{thm-restate}
\usepackage[margin=1cm]{caption}
\usepackage{mathtools}
\usepackage[colorlinks=true, linkcolor=blue, citecolor=red, pdfborder={0 0 0}]{hyperref}
\usepackage{cleveref}
\usepackage{accents}
\usepackage{wrapfig}
\usepackage{multirow}
\usepackage{float}
\usepackage{blkarray}
\usepackage{tikz}
\usepackage{yfonts}
\usepackage{bm}
\usepackage{stmaryrd}

\theoremstyle{plain}
 \newtheorem{theorem}{Theorem}[section]
 \newtheorem{lemma}[theorem]{Lemma}
 \newtheorem{prop}[theorem]{Proposition}
 \newtheorem{cor}[theorem]{Corollary}
 \newtheorem{conj}[theorem]{Conjecture}

\theoremstyle{definition}
 \newtheorem{defn}[theorem]{Definition}
 \newtheorem{assumption}[theorem]{Assumption}

 \newtheorem{claim}[theorem]{Claim}

\theoremstyle{remark}
 \newtheorem{remark}[theorem]{Remark}
 
\numberwithin{equation}{section}
\numberwithin{theorem}{section}

\newcommand\nc\newcommand
\nc\dmo\DeclareMathOperator
\nc\bs\boldsymbol
\nc{\labbr}[1]{{\sc\lowercase{#1}}}
\nc{\abbr}[1]{{\sc\uppercase{#1}}}
\nc{\ignore}[1]{}

\nc{\cf}{cf.\ }
\nc{\ie}{i.e.\ }
\nc{\eg}{e.g.\ }

\nc{\overeq}[1]{\stackrel{\text{\tiny #1}}{=}}
\nc{\overleq}[1]{\stackrel{\text{\tiny #1}}{\le}}

\nc{\filldi}{{\hfill$\Diamond$}}

\dmo\per{per}

\nc{\red}[1]{{\color{red} #1}}
\nc{\blue}[1]{{\color{blue} #1}}
\nc{\green}[1]{{\color{green} #1}}
\nc{\cyan}[1]{{\color{cyan} #1}}
\definecolor{purple}{rgb}{0.9,0,0.8}
\nc{\purple}[1]{{\color{purple} #1}}
\definecolor{gray}{rgb}{0.5,0.5,0.5}
\nc{\gray}[1]{{\color{gray} #1}}
\nc{\note}[1]{{\blue{\textup{[#1]}}}}
\nc{\al}{\alpha} \nc{\Al}{\Alpha}
\nc{\be}{\beta} \nc{\Be}{\Beta}
\nc{\gam}{\gamma} \nc{\Gam}{\Gamma}
\nc{\eps}{\varepsilon}
\nc{\epp}{\epsilon}
\nc{\vep}{\varepsilon}
\nc{\de}{\delta} \nc{\De}{\Delta}
\nc{\lam}{\lambda} \nc{\Lam}{\Lambda}
\nc{\ka}{\kappa} \nc{\Ka}{\Kappa}
\nc{\vpi}{\varpi} \nc{\vph}{\varphi}
\nc{\om}{\omega} \nc{\Om}{\Omega}
\nc{\si}{\sigma} \nc{\Si}{\Sigma}

\nc{\A}{\mathbb{A}}
\nc{\B}{\mathbb{B}}
\nc{\C}{\mathbb{C}}
\nc{\D}{\mathbb{D}}
\nc{\E}{\mathbb{E}}
\nc{\F}{\mathbb{F}}
\nc{\G}{\mathbb{G}}
\nc{\bI}{\mathbb{I}}
\nc{\J}{\mathbb{J}}
\nc{\K}{\mathbb{K}}

\nc{\M}{\mathbb{M}}
\nc{\N}{\mathbb{N}}
\nc{\bO}{\mathbb{O}}

	\renewcommand{\P}{\mathbb{P}}
\nc{\Q}{\mathbb{Q}}
\nc{\R}{\mathbb{R}}
\nc{\bS}{\mathbb{S}}
\nc{\T}{\mathbb{T}}
\nc{\U}{\mathbb{U}}
\nc{\V}{\mathbb{V}}
\nc{\W}{\mathbb{W}}
\nc{\X}{\mathbb{X}}
\nc{\Y}{\mathbb{Y}}
\nc{\Z}{\mathbb{Z}}

\nc{\cA}{\mathcal{A}}
\nc{\cB}{\mathcal{B}}
\nc{\cC}{\mathcal{C}}
\nc{\cD}{\mathcal{D}}
\nc{\cE}{\mathcal{E}}
\nc{\cF}{\mathcal{F}}
\nc{\cG}{\mathcal{G}}
\nc{\cH}{\mathcal{H}}
\nc{\cI}{\mathcal{I}}
\nc{\cJ}{\mathcal{J}}
\nc{\cK}{\mathcal{K}}
\nc{\cL}{\mathcal{L}}
\nc{\cM}{\mathcal{M}}
\nc{\cN}{\mathcal{N}}
\nc{\cO}{\mathcal{O}}
\nc{\cP}{\mathcal{P}}
\nc{\cQ}{\mathcal{Q}}
\nc{\cR}{\mathcal{R}}
\nc{\cS}{\mathcal{S}}
\nc{\cT}{\mathcal{T}}
\nc{\cU}{\mathcal{U}}
\nc{\cV}{\mathcal{V}}
\nc{\cW}{\mathcal{W}}
\nc{\cX}{\mathcal{X}}
\nc{\cY}{\mathcal{Y}}
\nc{\cZ}{\mathcal{Z}}
\nc{\co}{\scriptsize$\cO$\normalsize}

\nc\notni{\not\owns}
\def\({\left(}
\def\){\right)}
\nc{\ii}{\mathrm{i}}
\renewcommand{\d}{\,d}
\nc{\ls}{\lesssim}
\nc{\gs}{\gtrsim}
\def \lf {\lfloor}
\def \rf {\rfloor}

\dmo{\lls}{\,\ls\,}
\dmo{\ggs}{\,\gs\,}
\nc{\cond}{\,\middle|\,}
\dmo{\I}{{I}}
\dmo{\II}{{II}}
\dmo{\argmin}{arg\,min}
\dmo{\argmax}{arg\,max}
\dmo{\area}{area}
\dmo{\diag}{diag}
\dmo{\diam}{diam}
\dmo{\esssup}{ess\,sup}
\dmo{\im}{im} % image
\dmo{\loc}{loc}
\dmo{\sgn}{sgn}
\dmo{\supp}{supp}
\dmo{\vol}{vol}
\dmo{\dist}{{dist}}

\DeclareFontFamily{U}{mathx}{\hyphenchar\font45}
\DeclareFontShape{U}{mathx}{m}{n}{
      <5> <6> <7> <8> <9> <10>
      <10.95> <12> <14.4> <17.28> <20.74> <24.88>
      mathx10
      }{}
\DeclareSymbolFont{mathx}{U}{mathx}{m}{n}
\DeclareFontSubstitution{U}{mathx}{m}{n}
\DeclareMathAccent{\widecheck}{0}{mathx}{"71}
\DeclareMathAccent{\wideparen}{0}{mathx}{"75}
\nc{\ol}{\overline}
\renewcommand{\t}{\tilde}
\nc{\ul}{\underline}
\nc{\wt}{\widetilde}
\nc{\wh}{\widehat}
\nc{\wch}{\widecheck}

\nc{\decto}{\downarrow}
\nc{\incto}{\uparrow}
\nc{\la}{\leftarrow}
\nc{\ra}{\rightarrow}
\nc{\La}{\Leftarrow}
\nc{\Ra}{\Rightarrow}
\nc{\lla}{\longleftarrow}
\nc{\lra}{\longrightarrow}
\nc{\Lla}{\Longleftarrow}
\nc{\Lra}{\Longrightarrow}
\nc{\iffa}{\Leftrightarrow}
\nc{\hra}{\hookrightarrow}
\nc{\inj}{\hookrightarrow}
\nc{\surj}{\twoheadrightarrow}
\nc{\bij}{\leftrightarrow}
\nc{\longbij}{\longleftrightarrow}
\nc{\dirlim}{\lim_{\lra}}
\nc{\invlim}{\lim_{\lla}}
\nc{\cra}{\stackrel{\sim}{\ra}}

\dmo{\adj}{adj}
\dmo{\Hom}{Hom}
\dmo{\proj}{proj}
\dmo{\tr}{tr}
\dmo{\Tr}{Tr}
\nc{\Span}{\operatorname{span}}

\dmo{\rank}{rank}

\dmo{\real}{Re}

\dmo{\Var}{Var}
\dmo{\Corr}{Corr}
\dmo{\Cov}{Cov}
\dmo{\1}{\mathbf{1}}
\dmo{\ind}{\mathbbm{1}}
\nc{\eqd}{\stackrel{\text{\tiny $d$}}{=}}
\nc{\dto}{\stackrel{d}{\to}}

\nc{\iid}{\abbr{iid}}
\dmo{\Ber}{Ber}
\dmo{\Bin}{Bin}
\dmo{\Poi}{Poi}
\dmo{\Vol}{Vol}
\nc{\nick}[1]{{\color{purple} #1}}
\nc{\nickQ}[1]{{\bf \emph{\nick{[#1]}}}}
\nc{\rlim}{{\ol\rho}}
\nc{\rbi}{{\sigma}}
\nc{\rLZ}{{\rho}}
\nc{\rnew}{{\wt\sigma}}
\nc{\bG}{{\bs G}}
\nc{\Gnpr}{{\bG_{n,p}^{{\tiny{(r)}}}}}
\nc{\Fcrit}{{\cF_H^{crit}}}
\nc{\Fstar}{{\cF_H^\star}}
\nc{\cIspan}{\cI_{H^\star}^{span}}

\title[Upper tails for hypergraph homomorphism counts]{Upper tails for homomorphism counts in sparse random hypergraphs}

\author[N.\ A.\ Cook]{Nicholas A.\ Cook$^\star$}
\thanks{${}^{\star}$ Supported in part by NSF grant DMS-2154029.}
\address{$^\star$Department of Mathematics, Duke University, 120 Science Drive, Durham, NC 27710, USA}\email{nickcook@math.duke.edu}
\author[N.\ Nguyen]{Nguyen Nguyen$^{\ddagger}$}
\address{$^{\ddagger}$Department of Statistics, Stanford University, 390 Jane Stanford Way,
Stanford, CA 94305, USA. }\email{ngng@stanford.edu}

\date{\today}

\begin{document}

\begin{abstract}
The \emph{infamous upper tail problem} for $r$-uniform hypergraphs is to estimate the probability that the number of copies of a fixed hypergraph $H$ in a large binomial $r$-uniform hypergraph $\boldsymbol{G}$ exceeds its expectation by a constant factor. The problem was popularized by Janson and Ruci\'nski in \cite{JaRu} and, particularly in the case of graphs ($r=2$), has been a driving example in the development of nonlinear large deviations theory. Recent work of the first author with Dembo and Pham \cite{CDP} has accomplished the \emph{na\"ive mean-field reduction step}, reducing the upper tail problem to an entropic variational problem on a space of weighted graphs. The latter was resolved for counts of $r$-uniform cliques and a certain linear 3-uniform hypergraph by Liu and Zhao in \cite{LiZh}, where they also conjectured a general formula. We confirm their conjecture for other classes of hypergraphs, including complete $r$-partite $r$-graphs, tight cycles, and the Fano plane. We also prove a general large deviation upper bound for counts of $r$-graphs $H$ satisfying certain edge covering properties.  
\end{abstract}

\maketitle

%Indented subsection numbering in TOC
\let\oldtocsubsection=\tocsubsection
\renewcommand{\tocsubsection}[2]{\hspace*{.0cm}\oldtocsubsection{#1}{#2}}
\let\oldtocsubsubsection=\tocsubsubsection
\renewcommand{\tocsubsubsection}[2]{\hspace*{1.8cm}\oldtocsubsubsection{#1}{#2}}

\setcounter{tocdepth}{2}
%\tableofcontents

\section{Introduction}
\label{sec:intro}

\subsection{Results}

A basic problem in probabilistic combinatorics is to estimate the probability that the number of copies of a small fixed graph in a large random graph differs significantly from its expectation. 
Since our interest is in the hypergraph version of the problem we set up the notation in this general setting before reviewing work on the graph case. 

For an integer $r\ge2$, an $r$-uniform hypergraph, or \emph{$r$}-graph, is a pair $H=(V(H),E(H))$, where $V(H)$ is a set and $E(H)\subseteq {V(H)\choose r}$ is a collection of $r$-sets in $V(H)$.
Elements of $V(H)$ are called vertices and elements of $E(H)$ are called edges. 

Given an $r$-graph $H=(V(H),E(H))$, a probability space $(\cX,\mu)$ and a symmetric function $g:\cX^r\to\R$, the $H$-homomorphism density in $g$ is defined
\begin{equation}
    \label{def:t}
    t(H,g):= \int_{\cX^{V(H)}} \prod_{e\in E(H)} g(x_e) \,d\mu_{V(H)}(x)
\end{equation}
where $x\mapsto x_e\in \cX^r$ is a coordinate projection to the coordinates in the $r$-set $e$ (the ordering of the coordinates being of no importance since $g$ is assumed to be symmetric), and for $S\subseteq V(H)$ we write $\mu_S$ for the $|S|$-fold product measure $\mu\otimes\cdots\otimes \mu$ on $\cX^S$.

For an $r$-graph $G$ over vertex set $V(G)=[n]$, we may view $G$ as a symmetric $\{0,1\}$-valued function on $[n]^r$, with $G(x_1,\dots, x_r)=1_{ \{x_1,\dots, x_r\}\in E(G)}$ (in particular $G(x_1,\dots, x_r)=0$ unless $x_1,\dots, x_r$ are all distinct). 
With $(\cX,\mu)=([n],\frac1n\#)$ the discrete interval with normalized counting measure, 
the $H$-density $t(H,G)$ coincides with the usual homomorphism density of $H$ in $G$:
\begin{equation}    \label{tHG}
t(H,G) =\frac1{n^{|V(H)|}}\sum_{x\in [n]^{V(H)}}\prod_{e\in E(H)} G(x_e)
\end{equation}
where $x_e:=(x_v)_{v\in e}$.

For $n\ge r$ and $p\in(0,1)$ we write $\bs G=\bs G_{n,p}^{(r)}$ for the binomial random $r$-graph over $V(\bs G) = [n]$, with random edge set $E(\bs G)$ such that the ${n\choose r}$ indicators $1_{\{x_1,\dots, x_r\}\in E(\bs G)}$ are jointly independent Bernoulli($p$) variables. 

The upper tail problem for homomorphism densities in $\bs G$ is concerned with asymptotics for 
\begin{equation}    \label{def:RH}
    R_H(n,p,\delta):= -\log \P\Big( t(H,\bs G) \ge (1+\delta)p^{|E(H)|}\Big)
\end{equation}
for $\delta>0$. This is an instance of the ``infamous upper tail'' that was popularized in the influential work \cite{JaRu}. We postpone a review of the literature to \Cref{sec:background}; see also the survey \cite[Section 7]{Chatterjee:survey} and references therein.

In many cases it has been shown that
\begin{equation}
    R_H(n,p,\delta) \asymp_\delta n^rp^{\Delta(H)}\log(1/p)
\end{equation}
assuming $p$ does not decay too quickly with $n$. (For our conventions on asymptotic notation see \Cref{sec:notation}.) 
Here we write $\Delta(H):=\max_{v\in V(H)} \deg_H(v)$ for the maximum degree of $H$, where $\deg_H(v):=|\{e\in E(H): v\in e\}|$ is the degree of a vertex $v\in V(H)$.
An upper bound of this order is obtained by computing the probability that $\bs G$ contains all ${m\choose r}$ edges over a fixed set of $m$ vertices, with $m=\lf C np^{\Delta(H)/r})\rf$ for $C=C(H,\delta)$ sufficiently large. (Note this construction requires $p\ge C' n^{-r/\Delta(H)}$ for the clique to be nonempty.)

Our problem here is to establish the finer asymptotic 
\begin{equation}    \label{main-goal}
    \frac{R_H(n,p,\delta)}{n^rp^{\Delta(H)} \log(1/p)}
    \longrightarrow \frac{\rlim_H(\delta)}{r!} 
\end{equation}
as $n\to\infty$, for an explicit function $\rlim_H:\R^+\to\R^+$ depending on $H$ (the normalization by $r!$ will be convenient below).
This problem has been thoroughly studied in the case $r=2$, where it has been shown \eqref{main-goal} holds with
\begin{equation}
    \label{def:rbi}
    \rlim_H(\delta) = 
    \rbi_H(\delta):=
    \begin{cases}
        \min\big\{\delta^{\Delta(H)/|E(H)|} \,,\, r\beta_H(\delta) \big\} & H \text{ regular} \\
        r\beta_H(\delta) & H \text{ irregular}
    \end{cases}
\end{equation}
as long as $\omega(n^{-1/\Delta(H)})\le p\le o(1)$ \cite{HMS,BaKa}.
($H$ is regular if $\deg_H(v)=\Delta(H)$ for all $v\in V(H)$.)
Here, $\beta=\beta_H(\delta)$ is the unique positive solution of the equation
\begin{equation}    \label{def:betaH}
    1+\delta=i_{H^\star }(\beta)
\end{equation}
where $i_{H^\star }(\cdot)$ is the independence polynomial for the induced subgraph $H^\star $ of $H$ on its vertices of degree $\Delta(H)$. 
An interesting feature of the rate function \eqref{def:rbi} in the regular case is that it is non-smooth at some point $\delta^*(H)>0$, indicating a phase transition in the structure of $\bs G$ conditional on the large deviation event.

For $r\ge3$ the picture is far from complete. 
The following result was obtained in \cite{LiZh} for a range of $p$, and extended to the stated ranges in \cite{CDP}. 

\begin{figure}
    \centering
    \includegraphics[width=0.3\linewidth]{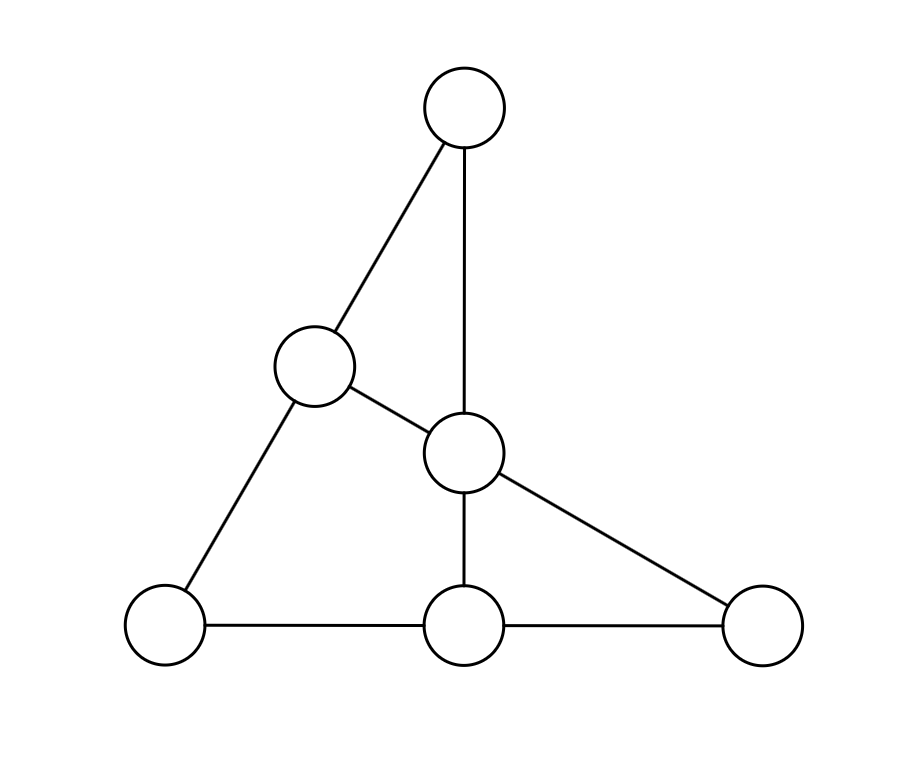}
    \includegraphics[width=0.3\linewidth]{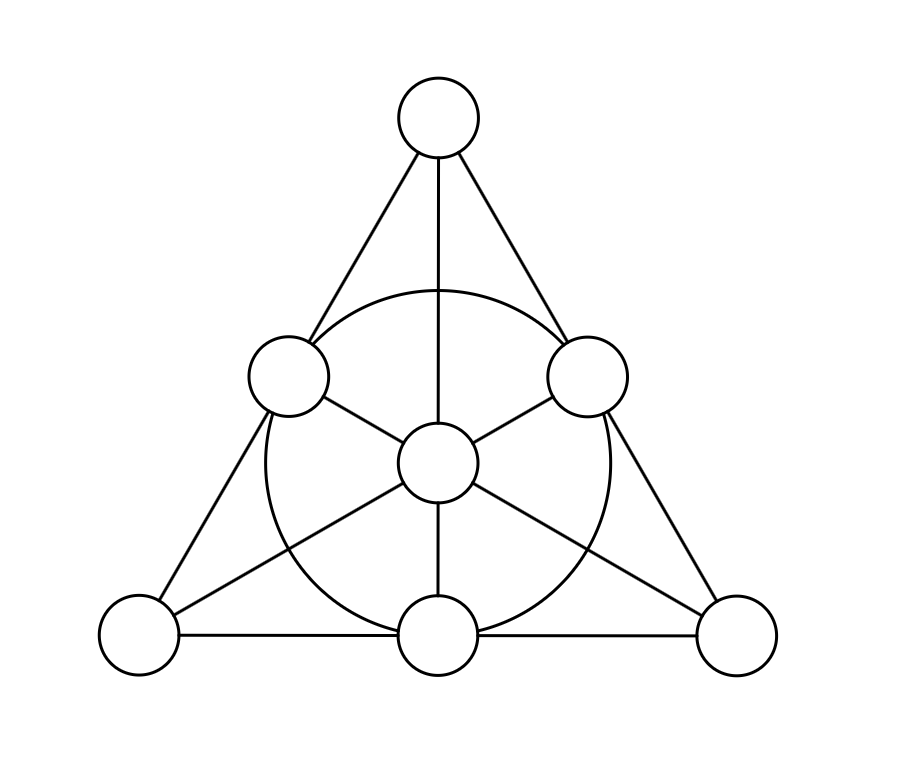}
    \includegraphics[width=0.3\linewidth]{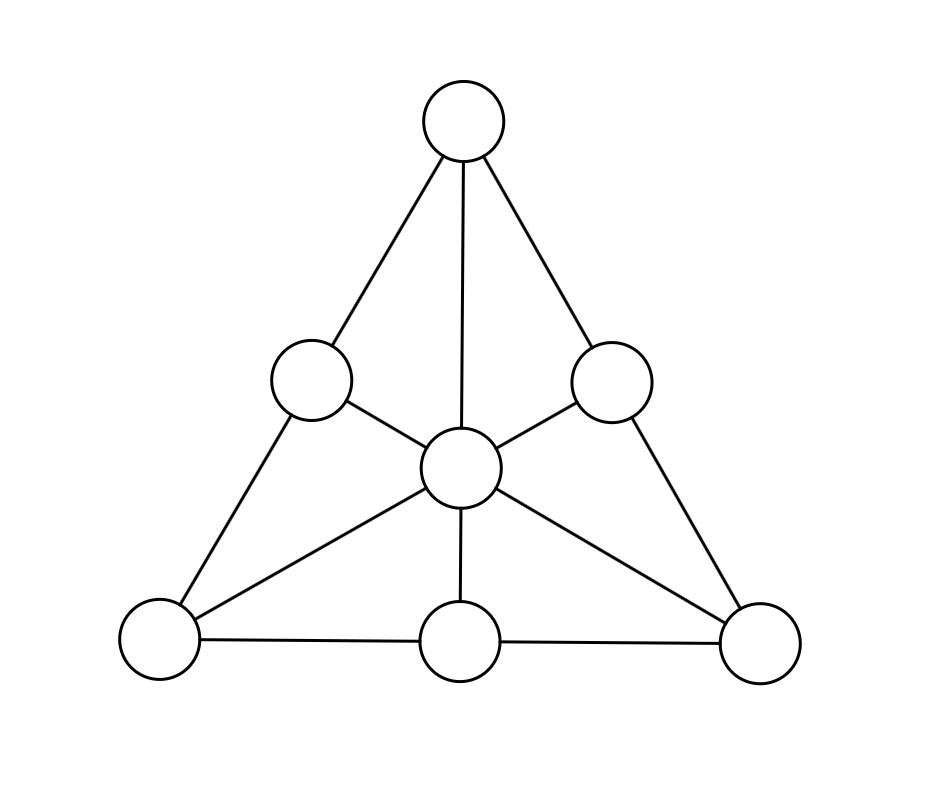}

    \caption{    3-graphs for which the large deviations rate function $\rlim_H$ is not given by the function $\rbi_H$ from \eqref{def:rbi}. 
    \emph{Left:} The 3-graph from \Cref{thm:LiZh}(b). The 6 vertices are represented by circles and the 4 edges by straight lines.
    \emph{Center:}
    The Fano plane from \Cref{thm:main}(c) is a 3-regular 3-graph with 7 vertices and 7 edges. Vertices are by represented small circles and edges by the 6 straight lines and one large circle. 
    \emph{Right:} (Also from \Cref{thm:main}(\ref{main.Fano})) the subgraph of the Fano plane obtained by removing a single edge.
    }
    \label{fig:Fano}
\end{figure}

\begin{theorem}[\hspace{-.01cm}{\cite{LiZh}}, {\cite[Corollary 3.2]{CDP}}]
\label{thm:LiZh}
Fix an $r$-graph $H$ and $\delta>0$. Assume $p=o(1)$.
\begin{enumerate}[(a)]
\item If $H=K_k^{(r)}$ is the $r$-uniform clique on $k$ vertices, 
and $p=\omega(n^{-c(r,k)})$ with $c(r,k) = 1/({k-1\choose r-1} + 1)$,
then \eqref{main-goal} holds with 
\begin{align}	
\rlim_H(\delta) =\rbi_H(\delta) &= \min\big\{ \delta^{r/k}, r\delta/k \big\} \,.
 \label{LiZh:asymp1}
\end{align}
\item If $H$ is the 3-graph on the left in Figure \ref{fig:Fano}, and $p=\omega(n^{-1/2})$, then \eqref{main-goal} holds with
\begin{align}	
\rlim_H(\delta)&=6 \min\Big\{ \sqrt{9+3\delta}-3, \sqrt{\delta} \Big\}  \,.
 \label{LiZh:asymp2}
\end{align}
\end{enumerate}
\end{theorem}

One verifies that \eqref{LiZh:asymp1} agrees with \eqref{def:rbi}, while \eqref{LiZh:asymp2} does not. 

In this work we establish the limit \eqref{main-goal} for more classes of $r$-graphs. We recall some terminology. 
For positive integers $m_1,\dots, m_r$ we write $K^{(r)}_{m_1,\dots, m_r}$ for a canonical copy of a complete $r$-partite $r$-graph on parts $V_1,\dots, V_r$ of respective sizes $m_1,\dots, m_r$. Thus, $K^{(r)}_{m_1,\dots, m_r}$ has vertex set $V_1\sqcup\cdots\sqcup V_r$, and each edge intersects each part at exactly one vertex. 

An $r$-graph $H$ is a \emph{tight cycle} on $\ell$ vertices if $V(H)$ can be identified with $[\ell]$ in such a way that $E(H)$ consists of all discrete intervals of the form $[v,v+r-1]\subset[\ell]$, where addition is taken mod $\ell$. See \Cref{fig:7cycle}. We write $C_\ell^{(r)}$ for a canonical copy of the $r$-uniform tight cycle on $\ell$ vertices. 

\begin{theorem}
    \label{thm:main}
    Fix an $r$-graph $H$ and $\delta, a>0$. 
    Assume $\Delta(H)\ge2$ and $\omega(n^{-a})\le p\le o(1)$.
    \begin{enumerate}[(a)]
        \item \label{main.r-partite}
        (Complete $r$-partite).
        Suppose $H$ is a regular complete $r$-partite $r$-graph.
        Thus, the $r$ vertex classes $V_1,\dots, V_r$ have common size $m\ge2$, and $\Delta(H)= m^{r-1}$. 
        If $a= 1/(m^{r-1}+1)$, then \eqref{main-goal} holds with 
        \begin{equation}    \label{rpartite:rH}
            \rlim_H(\delta) = \rbi_H(\delta)=
                 \min\bigg\{ \delta^{1/m}\,,\, r\bigg[ \Big( 1+\frac\delta{r}\Big)^{1/m} - 1\bigg]\bigg\} \,.
        \end{equation}
        If $H=K^{(r)}_{m_1,m_2,\dots, m_r}$ with $1\le m_1<m_2,\dots, m_r$ and $a=1/(m_2\cdots m_r+1)$ then
                \eqref{main-goal} holds with 
        \begin{equation}    \label{rpartite':rH}
            \rlim_H(\delta) = \rbi_H(\delta)=
            r[(1+\delta)^{1/m_1}-1 ]
            \,.
        \end{equation}
        
        \item \label{main.cycle}
        (Cycles).
        If $H=C_\ell^{(r)}$ is a tight $r$-uniform cycle on $\ell$ vertices for $\ell\in\cL_r:=\bigcup_{i=2}^\infty (i(r-1),ir]$, and $a=\frac1{r+1}$, then \eqref{main-goal} holds with 
        \begin{equation}    \label{cycle:rH}
            \rlim_H(\delta) = \rbi_H(\delta)
            =\min\big\{  \delta^{r/\ell}, r\beta_H(\delta) \big\}
        \end{equation}
        where $\beta_H(\delta)$ is as in \eqref{def:betaH}.
        If $H$ is any proper subgraph of $C_\ell^{(r)}$ with $\Delta(H)=r$, $\ell\in \cL_r$, and $a=\frac1{r+1}$, then \eqref{main-goal} holds with 
        \begin{equation}    \label{cycle:rHsub}
            \rlim_H(\delta) = \rbi_H(\delta)
            = r\beta_H(\delta) \,.
        \end{equation}     
        
        \item \label{main.Fano}
        (Fano plane).
        If $H$ is the Fano plane (see center of Figure \ref{fig:Fano}) and $a=\frac13$, then \eqref{main-goal} holds with $\rlim_H(\delta) = \min\{ \tfrac37\delta,\delta^{\frac 37}\}$.
        If $H$ is the Fano plane with one edge removed (right of \Cref{fig:Fano}), and $a=\frac13$, then \eqref{main-goal} holds with 
        $
            \rlim_H(\delta) = \min
            \{
            \frac{\delta}8,\frac{\sqrt{\delta}}2
            \}
        $.        
    \end{enumerate}
\end{theorem}

While the rate function $\rlim$ in parts (a) and (b) agree with \eqref{def:rbi}, the rate functions in part (c) do not.

\begin{remark}[Cycles case]
For part (b), note that $\cL_r$ can be expressed as a finite union
\[
\cL_r=\bigcup_{i=2}^{r-2}\big(i(r-1),ir\big] \cup \big((r-1)^2\vee 2,+\infty\big)
\]
In particular we cover all $\ell\ge3$ for $r=2$ and all $\ell>(r-1)^2$ for $r\ge3$.
We have
\begin{align*}
    \cL_2 &= [3,+\infty)\\
    \cL_3 & = [5,+\infty)\\
    \cL_4 &= \{7,8\}\cup[10,+\infty)\\
    \cL_5 &= \{9, 10\}\cup\{13,14,15\}\cup[17,+\infty)\\
    \cL_6 &= \{11,12\}\cup\{16,17,18\}\cup\{21,22,23,24\}\cup[26,+\infty).
\end{align*}
For $r=3$ the only missing case is $C_4^{(3)}=K_4^{(3)}$ that was covered by \Cref{thm:LiZh}(a).
\end{remark}

\begin{figure}[t]
    \centering
    \includegraphics[width=0.5\linewidth]{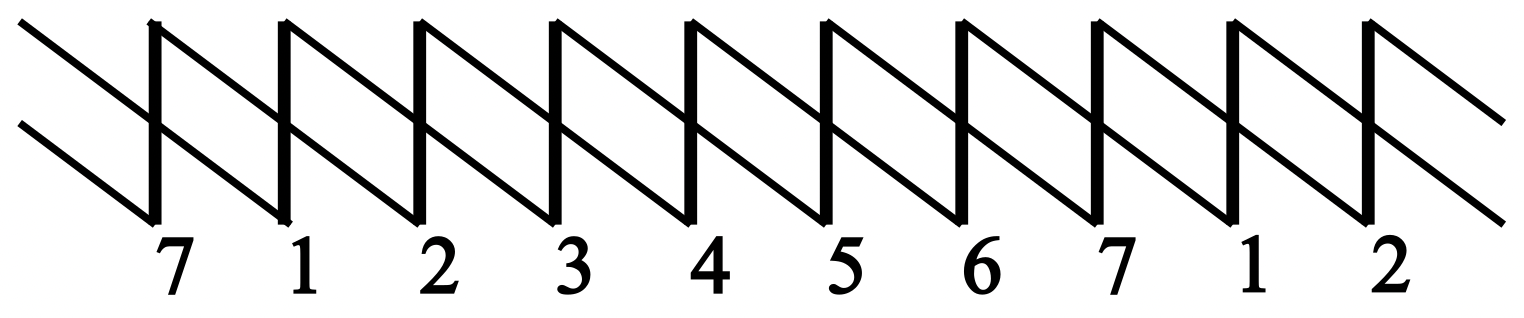}\\
    \vspace{.4cm}
    \includegraphics[width=0.5\linewidth]{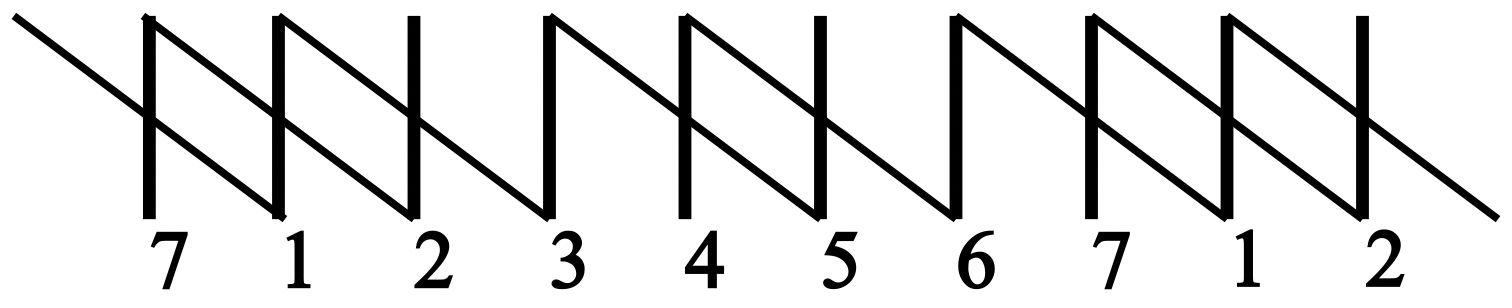}
    \caption{\emph{Top:} The tight 3-uniform 7-cycle $C_7^{(3)}$. Vertices are vertical lines, cyclically labeled on the number line, and edges are diagonal lines. 
    \emph{Bottom:} A subgraph of $C_7^{(3)}$in which vertex 1 has degree 3 and all other vertices have degree 2.}
    \label{fig:7cycle}
\end{figure}

\begin{remark}  \label{rmk:conditional}
    For the case $r=2$, results on the structure of $\bs G$ conditional on joint upper tail events $\bigcap_{i=1}^k\{t(H_i,\bs G)\ge (1+\delta_i)p^{|E(H_i)|}\}$ are obtained in \cite{CoDe:ergms} by a careful stability analysis of arguments from \cite{BGLZ}. Analogous results could be obtained by similar extensions of the proof of \Cref{thm:main}. 
    As the arguments are lengthy we do not pursue this here.
\end{remark}

The lower bounds on $R_H(n,p,\delta)$ (i.e. upper bounds on large deviation probabilities) for \Cref{thm:main}(\ref{main.r-partite},\ref{main.cycle}) 
are consequences of the following result for a general class of $r$-graphs satisfying  certain covering properties, the statement of which we defer to \Cref{sec:gen}.
Definitions of the fractional matching number $\nu^*(H)$ and transversals are recalled in \Cref{sec:frac}.

\begin{theorem} \label{thm:gen}
    Let $H$ be an $r$-graph satisfying \Cref{assu:good}. Assume $\Delta(H)\ge2$ and\\ $\omega(n^{-1/(\Delta(H)+1)})\le p\le o(1)$.
    Then for any fixed $\delta>0$,
    \begin{equation}    \label{bd:gen1}
        \frac{R_H(n,p,\delta)}{n^rp^{\Delta(H)}\log(1/p)} \ge \frac1{r!}\min\big\{\delta^{\Delta(H)/|E(H)|} \,,\, r\beta_H(\delta) \big\} + o(1).
    \end{equation}
    If we further assume either 
    \begin{enumerate}
        \item $\nu^*(H)>|E(H)|/\Delta(H)$, or 
        \item $H$ has a unique transversal  of size $|E(H)|/\Delta(H)$ and satisfies the covering properties (VG\ref{VG1},\ref{VG2}) of \Cref{def:good},
    \end{enumerate}
    then
    \begin{equation}    \label{bd:gen2}
            \frac{R_H(n,p,\delta)}{n^rp^{\Delta(H)}\log(1/p)} \ge  \frac{\beta_H(\delta)}{(r-1)!}  + o(1).
    \end{equation}
\end{theorem}

The assumptions of \Cref{thm:gen} cover (for instance) the $r$-graphs in \Cref{thm:main}(\ref{main.r-partite},\ref{main.cycle}), as well as all connected $2$-graphs (see \Cref{prop:2graph}), thus recovering results for that case from \cite{BGLZ,HMS,BaKa}. In these cases we can verify matching upper bounds for $R_H(n,p,\delta)$ by computing the probability that $\bG$ contains appropriate small planted structures that account for the extra copies of $H$. 
\Cref{thm:gen} also covers all complete $r$-partite $r$-graphs having at most one part consisting of a single vertex:

\begin{cor} \label{cor:rpartite}
    Let $H=K^{(r)}_{m_1,\dots, m_r}$ for $1\le m_1\le\cdots\le m_r$. Assume $m_2>1$ and $\omega(n^{-1/(\Delta(H)+1)})\le p\le o(1)$.
    Let $r'\le r$ be the largest number such that $m_1=m_{r'}$.
    Then for any fixed $\delta>0$,
    \begin{equation}
        \frac{R_H(n,p,\delta)}{n^rp^{\Delta(H)}\log(1/p)} \ge \frac1{r!}\min\bigg\{ \delta^{1/m_1}\,,\, r\bigg[ \Big( 1+\frac\delta{r'}\Big)^{1/m_1} - 1\bigg]\bigg\}  + o(1).
    \end{equation}
\end{cor}

Note that the case $r'=1$ was addressed in \eqref{rpartite':rH} (with a better bound).

We only get matching upper bounds on $R_H(n,p,\delta)$ for the cases in \Cref{thm:main}, and indeed we do not expect the lower bounds of \Cref{thm:gen} to be sharp in general.

\subsection{Background}
\label{sec:background}

In \cite{ChDe} Chatterjee and Dembo introduced a general framework for the study of large deviations for nonlinear functions on high-dimensional product spaces; see also the survey and book \cite{Chatterjee:survey,Chatterjee:book}. The framework splits the problem of establishing \eqref{main-goal} into two steps.

\subsubsection{Step 1: NMF approximation}

The first step is to show
\begin{equation}    \label{NMF}
    R_H(n,p,\delta)=(1+o(1))\Phi_H(n,p,\delta+o(1))
\end{equation}
where
\begin{equation}    \label{def:Phi}
    \Phi_H(n,p,\delta):= \inf_{Q\in [0,1]^{{[n]\choose r}}} \Big\{ I_p(Q): t(H,Q) \ge (1+\delta)p^{|E(H)|} \Big\}\,.
\end{equation}
Here, $Q$ ranges over \emph{weighted} $r$-graphs $Q:{[n]\choose r}\to [0,1]$, $e\mapsto q_e$, 
and, we have extended the homomorphism density formula $t(H,G)$ from \eqref{tHG} to weighted graphs.
The entropy function $I_p$ is defined
\begin{align}
I_p(Q)&:= \sum_{e\in {[n]\choose r}} I_p(q_e)
\label{ipQ}\\
I_p(q) &:= q\log\frac qp + (1-q)\log\frac{1-q}{1-p}\,,\qquad q\in[0,1]    \label{def:ipq}
\end{align}
(extended continuously to $0,1$). 
Thus, $I_p(q)$ is the relative entropy 
of the Bernoulli($q$) measure with respect to the Bernoulli($p$) measure on $\{0,1\}$.
By independence, $I_p(Q)$ is then the relative entropy of 
an \emph{inhomogeneous} random $r$-graph $\bG_Q$ for which the indicators $\{\bG_Q(e):e\in {[n]\choose r}\}$ are independent Bernoulli($q_e$) variables, with respect to $\Gnpr$.
Intuitively, $\exp(-I_p(Q))$ is roughly the probability that $\Gnpr$ lies ``close to'' a typical realization of $\bG_Q$, and the infimum in \eqref{def:Phi} finds the ``least unlikely'' inhomogeneous profile $Q$ under which $\bG_Q$ typically achieves $t(H, \bG_Q)\ge (1+\delta)p^{|E(H)|}$.

The asymptotic \eqref{NMF} is known as the \emph{na\"ive mean-field (NMF) approximation}, due to a connection with statistical physics (it is the simplest in a hierarchy of variational approximations for partition functions of disordered systems); see \cite{ChDe,Chatterjee:survey}. 

The NMF approximation \eqref{NMF} was established for the case of dense graphs ($r=2$ and $p$ fixed) in \cite{ChVa} using graphon theory. 
General methods to establish quantitative versions of \eqref{NMF} (applying to tails for general nonlinear functions on the discrete hypercube) allowing $p$ to decay like $n^{-c}$ for a positive constant $c>0$ depending on $H$ were developed in \cite{ChDe,Eldan:NMF}; see \cite{LiZh} for the application of results from \cite{Eldan:NMF} to the hypergraph setting. 
An alternative approach to \eqref{NMF} (as well as joint upper and lower tail probabilities) for general $r$ that allows faster decay of $p$ was developed in \cite{CDP}, based on effective versions of the regularity and counting lemmas for hypergraphs in the large deviations regime; see Theorem \ref{thm:CDP} below. Earlier works \cite{Augeri,CoDe} for the case $r=2$ relied on spectral methods that do not generalize to hypergraphs.
The upper-tail NMF approximation (in a slightly different but closely related form to \eqref{def:Phi}) has been established in an optimal range of $p$ for regular 2-graphs in \cite{HMS,BaBa}, using an elegant sharpening of the moment method from \cite{JOR}. See \cite{BaKa,CHMS} for recent extensions to irregular 2-graphs. 
We also note that the analogue of \eqref{NMF} for lower tails in the $r$-graph setting was established in an essentially optimal range of $p$ in \cite{KoSa}.
For $p$ decaying sufficiently fast with $n$ the lower tail is captured by the Poisson approximation rather than the NMF approximation \cite{Janson:LT, JaWa:LT}.
See also the remarkable recent work \cite{JPPS} on the rich behavior of the lower tail for triangle counts in the critical window $p\asymp n^{-1/2}$ separating the NMF and Poissonian regimes. 

In this work we will use the following result from \cite{CDP} accomplishing Step 1 for the binomial $r$-graph; it is a special case of a more general result covering joint upper and lower tails for homomorphism counts of a fixed collection of $r$-graphs $H_1,\dots, H_m$.
In addition to the maximum degree $\Delta(H)$, the theorem involves a hypergraph degree parameter $\Delta'(H)$ whose definition we do not recall here (see \cite[Equation (3.15)]{CDP}); we only note that 
\begin{equation}    \label{Delta'UB}
    \Delta'(H)\le \Delta(H)+1.
\end{equation}

\begin{theorem}[NMF approximation {{\cite{CDP}}}]
\label{thm:CDP}
Fix an $r$-graph $H$ and $\delta>0$. \\
If $\omega(n^{-1/\Delta'(H)})\le  p<1$ then
\begin{equation}    \label{LB:CDP}
R_H(n,p, \delta) \ge
(1+o(1)) \Phi_H(n,p,\delta+o(1))
\end{equation}
while if $\omega(n^{-1/\Delta(H)})\le p<1$, then
\begin{equation}    \label{UB:CDP}
R_H(n,p,\delta) \le 
(1+o(1)) \Phi_H(n,p,\delta+o(1))\,.
\end{equation}
\end{theorem}

\subsubsection{Step 2: NMF variational problem}

Having established \eqref{NMF}, it remains to extract an asymptotic formula from the high-dimensional optimization problem \eqref{def:Phi}. 
This is closely related to the classical \emph{extremal problem} of estimating
\begin{equation}    \label{extremal}
    \min_{G: {[n]\choose r}\to \{0,1\}}\bigg\{ \sum_{e\in{[n]\choose r}} G(e): t(H,G) \ge \lambda\bigg\}
\end{equation}
for given $\lambda>0$. That is, what is the smallest number of edges in a graph on $n$ vertices attaining a given number of $H$-homomorphisms?
This problem has a long history; see \cite{Alon0,FrKa,JOR}.
An estimate that is sharp in many cases is provided by a generalization of H\"older's inequality due to Finner (see \Cref{lem:Finner}), achieved by arranging the edges of $G$ in a clique on $\sim n\lambda^{1/|V(H)|}$ vertices. 

The NMF variational problem \eqref{def:Phi} is more complicated than the extremal problem \eqref{extremal}. 
Since for $p\le\frac12$, $I_p(q)$ attains its maximum value of $\log(1/p)$ at $q=1$, the optimizers in \eqref{extremal} are natural candidates for optimizers in \eqref{def:Phi}; however, competing strategies emerge with some intermediate edge weights. Indeed, whereas in \eqref{extremal} we try to attain a large value for $t(H,\cdot)$ while minimizing the total number of edges, in \eqref{def:Phi} we seek the \emph{least unlikely} way to achieve a large value for $t(H,\cdot)$. 
Hence, we must also consider the contribution coming ``for free'' from edges $e$ with the typical value $q_e=p$, as well as contributions using some typical edges and some edges with large values. 
In fact, in cases where $\Phi_H(n,p,\delta)$ has been determined in the sparsity range $\omega(n^{-1/\Delta(H)})\le p\le o(1)$, 
the optimizers $Q=(q_e)$ only take on values at the extremes $\{p,1\}$, and can hence be viewed as a constant weighted graph $Q_0\equiv p$ (the expectation of $\bG$) that has been modified with small ``planted'' structures $\{e\in{[n]\choose r} : q_e=1\}$ that lead to a large boost in $t(H,Q)$. 

For the case of dense 2-graphs ($p$ fixed), asymptotic formulas for $\Phi_H(n,p,\delta)$ were obtained in a large range of $\{(p,q): 0\le p\le q\le 1\}$ in \cite{ChVa, LuZh:dense}, though the problem remains open in some ranges. 
For sparse 2-graphs, with $\omega(n^{-1/\Delta(H)})\le p\le o(1)$ (with $\Delta(H)$ the maximum degree of $H$), an essentially complete solution was obtained in \cite{LuZh:sparse,BGLZ}, with results for certain $H$ in the sparser range $\omega(n^{-2/\Delta(H)})\le p\le o(n^{-1/\Delta(H)})$ \cite{BGLZ,HMS,BaBa}. 
In \cite{HMS,CoDe:ergms,BaKa}, \emph{stability estimates} for optimizers were obtained, in order to prove results on the structure of sparse binomial graphs conditioned on upper-tail events. 

We also mention work \cite{BhDe,Gunby} on large deviations for homomorphism counts in random $d$-regular 2-graphs, random directed graphs \cite{Park:digraphs}, and for \emph{induced} subgraph counts in binomial 2-graphs \cite{Cohen22}, where optimizers in \eqref{def:Phi} can take on intermediate values $q_e\in(p,1)$.  
The analogous NMF variational problem for lower tails was studied in \cite{Zhao:lower}, with stability estimates and conditional structure results obtained in recent work \cite{Chin:lower}.

In \cite{LiZh}, Liu and Zhao state a conjecture for the leading order asymptotic for $\Phi_H(n,p,\delta)$, in terms of a quantity $\rLZ_H(\delta)$ obtained as the solution to a finite-dimensional variational problem, whose definition is deferred to \Cref{sec:SL}.

\begin{conj}[Liu--Zhao {\cite[Conjecture 3.5]{LiZh}}]
\label{conj:LiZh}
Fix an $r$-graph $H$ and $\delta>0$. If $\omega(n^{-1/\Delta(H)})\le p\le o(1)$, then
\begin{equation}    \label{asymp:LiZh}
    \Phi_H(n,p,\delta)=  (1+o(1))\frac1{r!}\rLZ_H(\delta) n^rp^{\Delta(H)}\log(1/p)\,,
\end{equation}
where $\rho_H(\delta)$ is defined in \eqref{def:rhoH}.
\end{conj}

Liu and Zhao establish \Cref{conj:LiZh} for the $r$-graphs of \Cref{thm:LiZh}.
In the present work we establish the conjecture for the $r$-graphs of \Cref{thm:main}.

It is shown in \cite[Lemma 3.6]{LiZh} that \eqref{asymp:LiZh} holds as an upper bound:
\begin{equation}
    \label{UB:LiZh}
    \Phi_H(n,p,\delta)\le  (1+o(1))\frac1{r!}\rLZ_H(\delta) n^rp^{\Delta(H)}\log(1/p)
\end{equation}
assuming $\omega(n^{-1/\Delta(H)})\le p\le o(1)$ and $\delta>0$ is fixed.
Combining with \eqref{UB:CDP} we obtain the following:

\begin{prop}\label{prop:upper}
    With hypotheses as in \Cref{conj:LiZh}, 
    \[
    R_H(n,p,\delta) \le (1+o(1))\frac1{r!}\rLZ_H(\delta) n^rp^{\Delta(H)}\log(1/p)\,.
    \]
\end{prop}

For each case of \Cref{thm:main}, our remaining tasks are hence to:
\begin{enumerate}
    \item compute $\Delta'(H)$ to verify the condition on the sparsity exponent $a$;
    \item verify $\rLZ_H(\delta)= \rlim_H(\delta)$; 
    \item establish the matching lower bound in \eqref{UB:LiZh}.
\end{enumerate}

\subsection{Organization of the paper}

In \Cref{sec:notation} we summarize conventions on asymptotic notation, graphs and homomorphism densities that are used throughout the paper.
In \Cref{sec:prelim} we collect basic estimates on relative entropies and homomorphism densities, along with a lemma on optimization of linear functionals over non-convex domains. 
In \Cref{sec:labelings} we recall key definitions from \cite{LiZh}, in particular the function $\rLZ(\delta)$ from \Cref{conj:LiZh}, along with some basic facts from fractional hypergraph theory. We also introduce the key notion of a \emph{strict stable labeling} that is crucial for isolating leading order contributions to homomorphism densities. 
In \Cref{sec:gen} we establish \Cref{prop:gen} giving a lower bound on $\Phi_H(n,p,\delta)$ (yielding the large deviations upper bound of \Cref{thm:gen}) for a general class of $r$-graphs $H$.
Parts (\ref{main.r-partite}) and (\ref{main.cycle}) of \Cref{thm:main} are established in Sections \ref{sec:rpartite} and \ref{sec:cycles}, respectively, using \Cref{prop:gen}.
In \Cref{sec:Fano} we prove \Cref{thm:main}(\ref{main.Fano}) on the homomorphism densities for the Fano plane.

\subsection{Notational conventions}
\label{sec:notation}

\subsubsection{Asymptotic notation}

Implicit constants in asymptotic notation $O(\,\cdot\,)$, $\ls$ are absolute. Dependence on fixed parameters such as $H$ is indicated with subscripts, e.g.\ $O_H(\,\cdot\,)$, $\ls_H$. 
For real $a,b$ we write $a\asymp b$ to mean $a\ls b\ls a$, and $a\lessapprox b$ to mean $|a|\le \log^{O(1)}(1/p) b$, indicating dependence of the implicit constant in the exponent of the log with subscripts.

For the majority of the proofs we work in the general setting of homomorphism densities over probability spaces, as in \eqref{def:t}, and hence there will be no asymptotic parameter $n$. 
Thus:\\

\emph{For the remainder of the paper, asymptotic notation $o(\,\cdot\,), \omega(\,\cdot\,)$ is with respect to the limit $p\to0$.}\\

The parameter $n$ only reappears at the conclusions of the proofs of \Cref{thm:main}(\ref{main.r-partite},\ref{main.cycle},\ref{main.Fano}), where we assume $n=\omega(1)$.

\subsubsection{Graphs and graph parameters}

For $S\subseteq V(H)$, the vertex neighborhood of $S$ is denoted
$N_H(S)=\{ v\in V(H): v\sim u \text{ for some }u\in S\}$.
We write $N_H(u):=N_H(\{u\})$, $\deg_H(u):= |N_H(u)|$, and $\Delta(H):= \max_{u\in V(H)} \deg_H(u)$.

We write $\cI_H$ for the collection of independent sets in $H$, that is, sets $S\subset V(H)$ such that $|e\cap S|\le 1$ for all $e\in E(H)$. The \emph{independence polynomial} for $H$ is 
\begin{equation}
    \label{def:IP}
    i_H(x) = \sum_{S\in \cI_H}x^{|S|}.
\end{equation}
By convention we always have $\emptyset\in \cI_H$, so $i_H(0)=1$.

For an $r$-graph $H=(V,E)$ and $U\subset V$ we write $H[U]$ for the induced subgraph of $H$ on vertex set $U$. That is, $V(H[U])=U$ and $E(H[U]) = \{ e\in E(H): e\subseteq U\}$.
Let 
\begin{equation}    \label{def:Hstar}
    V^\star(H):=\{v\in V(H): \deg_H(v)=\Delta(H)\}
    \,,\quad
    H^\star  := H[V^\star(H)]\,.
\end{equation}
That is, $H^\star $ is the induced subgraph of $H$ on the vertices of maximum degree.

We write $F\subseteq H$ to mean that $F$ is a graph on the same vertex set of $H$ with $E(F)\subseteq E(H)$. 

An $r$-graph $H$ is \emph{linear} if $|e\cap e'|\le 1$ for all distinct pairs $e,e'\in E(H)$.
$H$ is a copy of the \emph{loose path} $P_\ell^{(r)}$ of length $\ell$ if $H$ is a linear $r$-graph with $\ell$ edges that can be labeled $e_1,\dots, e_\ell$ so that $e_i\cap e_j=\emptyset$ unless $|i-j|=1$. 
%That is, $H$ has $\ell(r-1)+1$ vertices, with $\ell-1$ vertices $w_1,\dots, w_{\ell-1}$ having degree $2$, all other vertices having degree 1, and each edge containing 1 or 2 of the $w_i$. 

\subsubsection{Homomorphism densities}

As in previous works \cite{LuZh:sparse,BGLZ,LiZh}, our arguments to lower bound $\Phi_H(n,p,\delta)$ apply for homomorphism densities over arbitrary probability spaces. 
With a fixed probability space $(\cX,\mu)$ we use the following notation.
For a set $V$ we write $\mu_V$ for the product measure $\bigotimes_{v\in V} \mu$ on $\cX^V$. 
We also write $\mu^r:= \mu_{[r]}$.
We write $\|\cdot\|_{L^q}$ for the $L^q(\cX^r,\mu^r)$ norm.

Recall \eqref{def:t}.
For an $r$-graph $H=(V,E)$ and sets $S,T\subset V$, $A,B\subset\cX$ we sometimes write, e.g.
\begin{equation}    \label{def:tH-restricted}
    t(H, g \,;\, S\to A,T\to B):= \int_{\cX^V} \prod_{s\in S}1_{x_s\in A}\prod_{t\in T}1_{x_t\in B} \prod_{e\in E} g(x_e) d\mu_V(x). 
\end{equation}
We also abbreviate, e.g.
\begin{align*}
    t(H, g \,;\, v\to A) &:= t(H, g \,;\, \{v\} \to A) \quad \text{ for } v\in V\,,\\
    t(H, g \,;\, S\not\to A) &:=  t(H, g \,;\, S\to \cX\setminus A)\,.
\end{align*}
We will repeatedly use the homogeneity of homomorphism densities without mention: for any symmetric measurable $f:\cX^r\to\R$ and $a\in \R$, 
\begin{equation}    \label{homogeneity}
    t(H, af) = a^{|E(H)|} t(H, f). 
\end{equation}
Thus, for instance, $t(H, f) = o(p^{|E(H)|})$ if and only if $t(H,f/p) = o(1)$. 

Given a symmetric measurable function $f:\cX^r\to\R$, we write
\begin{equation}    \label{def:df}
    d_f(x):= \int_{\cX^{r-1}} f(x, x_1,\dots, x_{r-1}) d\mu^{r-1}(x_1,\dots, x_{r-1})
\end{equation}
which can be interpreted as a normalized degree for the ``vertex'' $x$. 
For $b\in \R$ we further denote sets
\begin{equation}    \label{def:Db}
    D_{\ge b} = D^f_{\ge b}:= \{ x\in \cX: d_f(x)\ge b\}\,,
    \qquad 
    D_{<b}= D^f_{<b}:= \cX\setminus D_{\ge b}\,.
\end{equation}

\section{Preliminary estimates}
\label{sec:prelim}

\subsection{Relative entropy}
\label{sec:entropy}

Recall
\[
I(x):= x\log(1-x)+ (1-x)\log(1-x)\,, \quad x\in[0,1]
\]
(extended continuously to the enpoints with $I(0)=I(1):=0$), and
\[
I_p(x):= x\log\frac xp + (1-x)\log\frac{1-x}{1-p}\,.
\]
For a probability space $(\cX,\mu)$ and a measurable function $g:\cX^r\to [0,1]$ we use the shorthand
\begin{equation}    \label{def:Ipg}
    I_p(g):= \int_{\cX^r} I_p\circ g\, d\mu^r\,.
\end{equation}
For an $r$-graph $H$, $p\in(0,1)$ and $\delta\ge 0$ we define
\begin{equation}
    \label{def:phiH.mu}
    \phi_H^{(\cX,\mu)}(p,\delta):= \inf\big\{ I_p(g): t(H,g/p)\ge 1+\delta\big\}
\end{equation}
where the infimum runs over symmetric measurable $g:\cX^r\to[0,1]$.
Note that with $\cX=[n]$ and $\mu=\frac1n\#$ the normalized counting measure, we have
\begin{equation}
    \label{LB:Phi-phi}
    \Phi_H(n,p,\delta) \ge n^{r} \phi_H^{(\cX,\mu)}(p,\delta)
\end{equation}
for all $p\in(0,1]$ and $\delta\ge0$. 
Indeed, the infimum on the left hand side runs over the subset of functions $g$ that vanish on the diagonal. 

Following \cite{LiZh}, it will be convenient to constrain the range of $g$ to $[p,1]$ and to have a gap above $p$. 
For small $\kappa>0$, 
let $\wt\phi_{H,\ka}^{(\cX,\mu)}(p,\delta)$ be defined as in \eqref{def:phiH.mu} but with the additional constraint that $g=p+f$ for $f$ taking values in $\{0\}\cup[\kappa p, 1-p]$.

\begin{lemma}[\hspace{-.01cm}{\cite[Lemma 6.1]{LiZh}}]
\label{lem:phi-tphi}
We have
\begin{equation}    \label{LB:phi-tphi}
    \phi_H^{(\cX,\mu)}(p,\delta) \ge \wt\phi_{H,\ka}^{(\cX,\mu)}(p,\delta')
    \qquad \text{ for }\; \delta'= \delta + O_H((1+\delta)\kappa).
\end{equation}
\end{lemma}

\begin{proof}
Given a symmetric measurable function $g:\cX^r\to[0,1]$, let $f:\cX^r\to[0,1-p]$ be defined 
\[
f(x) = \begin{cases}
    g(x)-p & g(x)\ge (1+\kappa)p\\ 0 & \text{otherwise.}
\end{cases}
\]
Then $f$ satisfies \eqref{assu:kappa}. 
Since $I_p$ is decreasing on $[0,p]$ we have $I_p(p+f)\le I_p(g)$. 
Moreover, from the pointwise bound $g\le (1+\kappa) (p+f)$ it follows that 
\[
t(H,g)\le (1+\kappa)^{|E(H)|} t(H,p+f).
\]
Hence, if $t(H,g/p)\ge 1+\delta$, then
\[
t(H, 1+ f/p) \ge (1+\kappa)^{-|E(H)|}(1+\delta) = 1+ O_H((1+\delta)\kappa)
\]
and \eqref{LB:phi-tphi} follows. 
\end{proof}

With \Cref{lem:phi-tphi} in hand, we will often restrict attention to symmetric measurable $f$ on $\cX^r$ satisfying 
\begin{equation}
\label{assu:kappa}
    f(x)\in \{0\}\cup [ \kappa p, 1-p]\quad \forall x\in\cX^r \,\quad\; \text{ with } \kappa:= \frac1{\log(1/p)}
    \,.
\end{equation}

For convenience we define
\begin{equation}    \label{def:Jp}
    J_p(x):= \frac{I_p(p+x)}{I_p(1)}\,,\quad x\in [-p, 1-p]. 
\end{equation}
(Note $I_p(1)=\log(1/p)$.)
Similarly to \eqref{def:Ipg} we abbreviate
\begin{equation}
    J_p(f):= \int_{\cX^r} J_p\circ f\, d\mu^r
\end{equation}
for measurable $f:\cX^r\to [-p,1-p]$.
Note that $J_p(0)=0$ and $J_p(1-p)=1$.
Moreover, from \cite[Lemma 3.10]{CoDe},
\begin{equation}
    \label{Jp-quad}
    J_p(x) \ge x^2 \qquad \forall p\in [0,c], x\in [0,1-p]
\end{equation}
for a universal constant $c>0$. 
From \eqref{Jp-quad} it follows that if $f:\cX^r\to[0,1-p]$ satisfies
\begin{equation}
    \label{assu:K}
    J_p(f) \le Kp^\Delta
\end{equation}
for some $K,\Delta$, then for any $q\ge2$,
\begin{equation}
    \label{bd:Lq}
    \|f\|_{L^q}^q \le \|f\|_{L^2}^2 \le Kp^\Delta.
\end{equation}
Moreover, since one can show $J_p(x)\sim x$ for $x$ away from the boundary of $(0,1-p)$, the assumption \eqref{assu:kappa} of a gapped range yields control in $L^1$. 
We hence have the following lemma. Recall that $\lessapprox$ hides factors of size $\log^{O(1)}(1/p)$. 

\begin{lemma}[\hspace{-.01cm}{\cite[Lemma 6.3]{LiZh}}]
    \label{lem:L1-bound}
    For a symmetric measurable function $f$ on $\cX^r$ satisfying \eqref{assu:kappa} and \eqref{assu:K} with $K\lessapprox1$,
    \begin{equation}
        \|f\|_{L^1} \ls \frac{K}\kappa p^\Delta \log(1/p)
        \lessapprox p^\Delta.
    \end{equation}
\end{lemma}

\subsection{Homomorphism densities}
\label{sec:prelim-bds}

Now let $(\cX,\mu)$ be a probability space. 
Here we collect estimates on homomorphism counts from an $r$-graph $F$ into a symmetric measurable function $f:\cX^r\to[0,1-p]$.
Recall that asymptotic notation $o(1)$ is with respect to the limit $p\to0$.

The following generalized form of H\"older's inequality due to Finner \cite{Finner} has been a fundamental tool in the study of the upper tails problem for subgraph counts.

\begin{lemma}[Finner's inequality]
    \label{lem:Finner}
    Let $\cV$ be a finite set, and for each $v\in \cV$ let $(\cX_v,\mu_v)$ be a probability space. For $A \subseteq \cV$, let $(\cX^A,\mu_A) := (\prod_{v \in A} \cX_v, \otimes_{v\in A}\mu_v)$ be the associated product space.
    Suppose $\cE$ is a multiset of subsets of $\cV$, and $(\lam_e)_{e\in \cE}\in [0,1]^{\cE}$ are weights such that 
    \begin{equation}
        \sum_{e\in \cE} \lam_e1_{v\in e}  \le 1\qquad \forall v\in \cV\,.
    \end{equation}
    Then for any functions $f_e \in L^{1/\lam_e} (\cX^e, \mu_e), e\in\cE$, we have
    \[
    \int_{\cX^V}\prod_{e\in \cE} f_e \, d\mu_\cV \le \prod_{e\in\cE} \|f_e\|_{L^{1/\lam_e}(\cX^e, \mu_e)}\,.
    \]
\end{lemma}

We will often apply the following simple consequence of \Cref{lem:Finner}.

\begin{cor}
    \label{cor:Finner}
    Let $F$ be an $r$-graph of maximum degree $\Delta$, let $(\cX,\mu)$ be a probability space, and let $f:\cX^r\to \R$ be a symmetric measurable function. Then
    \begin{equation}
        t(F,f) \le \|f\|_{L^\Delta}^{|E(F)|}.
    \end{equation}
\end{cor}

\begin{proof}
    Apply \Cref{lem:Finner} with weights $\lam_e\equiv \frac1\Delta$.
\end{proof}

The following will be used to control the contributions of subgraphs $F\subset H$ with $\Delta(F)<\Delta(H)$. 
\begin{lemma}
    \label{lem:smallDelta}
    Let $F$ be an $r$-graph of maximum degree $\Delta_0$ and let $f$ satisfy $J_p(f)\lessapprox p^\Delta$ and \eqref{assu:kappa}. 
    Then
    \[
    t(F,f/p) \lessapprox_F p^{|E(F)|(\frac\Delta{\Delta_0}-1)}\,.
    \]
\end{lemma}

\begin{proof}
From \Cref{cor:Finner} and \Cref{lem:L1-bound}, 
\begin{align*}
    t(F,f) 
    &\le \|f\|_{L^{\Delta_0}}^{|E(F)|}\le \|f\|_{L^1}^{|E(F)|/\Delta_0}\lessapprox_F p^{|E(F)|\Delta/\Delta_0}.
\end{align*}
The claim follows from homogeneity.
\end{proof}

Recall the notation \eqref{def:df}, \eqref{def:Db}.
We will often apply Markov's inequality as follows:
\begin{equation}    \label{D-Markov}
    \mu(D_{\ge b}) \le b^{-1}\|d_f\|_{L^1(\cX)}  = b^{-1}\|f\|_{L^1(\cX^r)}
\end{equation}
where the last equality is from Fubini's theorem. 

\begin{lemma}
\label{lem:edge-weight}
    Let $H$ be an $r$-graph of maximum degree $\Delta$ and let $f$ satisfy $J_p(f)\lessapprox p^\Delta$ and \eqref{assu:kappa}.
    Let $e_0\in E(H)$.
    For some $s\in[r]$ select vertices $v_1,\dots, v_s\in e_0$, and let $\theta_1,\dots,\theta_s<1$ be fixed numbers such that  $\theta_1+\cdots+\theta_s>1$. 
    Then with $b_i:= p^{(1-\theta_i)\Delta}$, we have
    \[
    t(H, f/p \,;\, v_1 \to D_{\ge b_1},\dots, v_s \to D_{\ge b_s})=o(1).
    \]
\end{lemma}

\begin{proof}
    From \Cref{lem:Finner} (taking $1/\Delta$ for all weights) and \Cref{lem:L1-bound},
    \begin{align*}
        & t(H, f \,;\, v_1\to D_{\ge b_1},\dots, v_s\to D_{\ge b_s})\\
        &\quad \le \|f\|_{L^\Delta}^{|E(H)|-1}
        \bigg( \int_{D_{\ge b_1} \times\cdots\times D_{\ge b_s}\times \cX^{r-s}} f(x_1,\dots, x_r)^\Delta d\mu^r(x_1,\dots, x_r) \bigg)^{1/\Delta} \\
        &\quad\le \|f\|_{L^1}^{(|E(H)|-1)/\Delta} \big(\mu(D_{\ge b_1})\cdots \mu(D_{\ge b_s})\big)^{1/\Delta}\\
        &\quad\le \frac{ \|f\|_{L^1}^{(|E(H)|-1+ s)/\Delta}}{(b_1\cdots b_s)^{1/\Delta}}\\
        &\quad\lessapprox_{H,\theta_1,\dots,\theta_s} p^{|E(H)| -1 + \sum_{i=1}^s \theta_i}\,.
    \end{align*}
    The claim follows.
\end{proof}

\begin{lemma}
    \label{lem:low-to-low}
    Let $f:\cX^r\to[0,1]$ satisfy $J_p(f)\lessapprox p^\Delta$ and \eqref{assu:kappa}.
    Let $F$ be an $r$-graph of maximum degree $\Delta$.
    For any fixed $\eps>0$ and any $v\in V(F)$ with $\deg_F(v)<\Delta$,
    \[
    t(F,f/p\,;\, v\to D_{\ge p^{1-\eps}}) = o(1).
    \]
\end{lemma}

\begin{proof}
    Write $b=p^{\be -\eps}$ (we see any $\be\ge1$ works at the end).
    Let $F'$ be the subgraph of $F$ obtained by deleting $v$ and the edges containing it, and let $\Delta_0=\Delta(F')$. 
    Applying Markov's inequality as in \eqref{D-Markov}, followed by \Cref{cor:Finner} and \Cref{lem:L1-bound}, we have 
    \begin{align*}
        t(F,f\,;\, v\to D_{\ge b})
        \le \mu(D_{\ge b}) t(F',f) 
        &\le b^{-1}\|f\|_{L^1} \|f\|_{L^{\Delta_0}}^{|E(F')|}\\
        &\le b^{-1} \|f\|_{L^1}^{1+ |E(F')|/\Delta_0}
    \lessapprox_F p^{\eps-\be + \Delta( 1+|E(F')|/\Delta_0)}.
    \end{align*}
    Now bounding $|E(F')|\ge |E(F)|-\deg_F(v)\ge |E(F)|-\Delta+1$ and $\Delta_0\le \Delta$, we get 
    \[
    t(F,f\,;\, v\to D_{\ge b}) \lessapprox_F p^{\eps -\be + \Delta + |E(F)|-\Delta +1 }
    = p^{\eps -\be +|E(F)|+1}.
    \]
    Thus $ t(F,f\,;\, v\to D_{\ge b})=o(1)$ if $\be\ge1$, as desired.
\end{proof}

Recall that the loose path $P_\ell^{(r)}$ of length $\ell$ has $\ell$ edges and $\ell(r-1)+1$ vertices, with $\ell-1$ vertices $w_1,\dots, w_{\ell-1}$ having degree $2$, all other vertices having degree 1, and each edge containing 1 or 2 of the $w_i$. 

\begin{lemma}
    \label{lem:path}
    Let $\ell\ge2$.
    For any $f:\cX^r\to[0,1]$ and $b_1,\dots, b_{\ell-1}>0$, 
    \[
    t(P_\ell^{(r)}, f \,;\, w_1\to D_{<b_1}, \dots, w_{\ell-1}\to D_{<b_{\ell-1}})
    < b_1\dots, b_{\ell-1} \|f\|_{L^1}\,.
    \]
\end{lemma}

\begin{proof}
    Label the vertices of $P_\ell^{(r)}$ by $[\ell(r-1)+1]$ so that the edges are $e_i = \{(i-1)(r-1)+1,\dots,(i-1)(r-1)+r\}$ for $1\le i\le \ell$, with $w_i= (i-1)(r-1)+r$ for $1\le i\le \ell-1$.
    
    We proceed by induction on $\ell$. For $\ell=2$ we have
    \begin{align*}
        t(P_2^{(r)}, f \,;\, w_1\to D_{<b_1})
        &= \int_{\cX^{2r-1}} 1_{x_r\in D_{<b_1}} f(x_1,\dots,x_r) f(x_r,\dots, x_{2r-1}) d\mu^{2r-1}(x)\\
        &= \int_{D_{<b_1}} d_f(x_r)^2 d\mu(x_r)\\
        &<b_1 \int_{\cX} d_f(x_r)d\mu(x_r)\\
        &= b_1\|f\|_{L^1}
    \end{align*}
    as desired.
    For $\ell\ge3$, 
    \begin{align*}
        &t(P_\ell^{(r)}, f \,;\, w_1\to D_{<b_1}, \dots, w_{\ell-1}\to D_{<b_{\ell-1}})\\
        &\quad = \int_{\cX^{\ell(r-1)+1}} \prod_{i=1}^{\ell-1} 1_{x_{w_i}\in D_{<b_i}} \prod_{j=1}^\ell f(x_{e_j}) d\mu^{\ell(r-1)+1}(x)\\
        &\quad=\int_{\cX^{(\ell-1)(r-1)+1}} \prod_{i=1}^{\ell-1} 1_{x_{w_i}\in D_{<b_i}} d_f(x_{w_{\ell-1}}) \prod_{j=1}^{\ell-1} f(x_{e_j}) d\mu^{(\ell-1)(r-1)+1}(x)\\
        &\quad< b_{\ell-1} \int_{\cX^{(\ell-1)(r-1)+1}} \prod_{i=1}^{\ell-1} 1_{x_{w_i}\in D_{<b_i}}  \prod_{j=1}^{\ell-1} f(x_{e_j}) d\mu^{(\ell-1)(r-1)+1}(x)\\
        &\quad = b_{\ell-1} t(P_{\ell-1}^{(r)}, f \,;\, w_1\to D_{<b_1}, \dots, w_{\ell-2}\to D_{<b_{\ell-2}})
    \end{align*}
    and the claim follows by induction.
\end{proof}

\subsection{Optimization}

We will need the following generalization of \cite[Lemma 5.4]{BGLZ}.

\begin{lemma}
\label{lem:cvx}
Let $h_1,\dots, h_d$ be convex nondecreasing functions on $[0,+\infty)$ and let $a,c_1,\dots,c_d>0$. The minimum of $F(x)=c_1x_1+\cdots + c_dx_d$ over the region $U=\{x\in [0,+\infty)^d: \sum_{k=1}^d h_k(x_k)\ge a\}$ is attained when all but one of the $x_k$ are equal to zero. 
\end{lemma}

\begin{proof}
    By rescaling coordinates we may assume without loss of generality that $c_1=\cdots=c_d=1$.
    Let $x\in U$ be arbitrary. It suffices to show there exists $k\in[d]$ such that $(x_1+\dots+ x_d)e_k\in U$ where $e_k$ is the $k$th canonical basis vector.
    If $x=0$ we are done, so assume $x_1+\cdots+x_d>0$. 
    Let $\theta_\ell:= x_\ell/(x_1+\cdots+x_d)$. 
    For each $\ell\in[d]$, since $h_\ell$ is convex we have
    \begin{align*}
            h_\ell(x_\ell) &= h_\ell( \theta_\ell (x_1+\cdots+x_d) ) \\
            &= h_\ell\bigg( \sum_{k=1}^d \theta_k \big[ (x_1+\cdots+x_d) \delta_{k,\ell} + 0 \cdot (1-\delta_{k,\ell}) \big] \bigg)\\
            &\le  \sum_{k=1}^d \theta_k \big[ h_k(x_1+\cdots+x_d)\delta_{k,\ell} + h_k(0)(1-\delta_{k,\ell})\big]\,.
    \end{align*}
    Summing over $\ell$, we have
    \begin{align*}
        a %\le \sum_{\ell=1}^d h_\ell(x_\ell)
        &\le \sum_{\ell=1}^d \sum_{k=1}^d \theta_k \big[ h_k(x_1+\cdots+x_d)\delta_{k,\ell} + h_k(0)(1-\delta_{k,\ell})\big]\\
        &= \sum_{k=1}^d \theta_k \bigg[ h_k(x_1+\cdots+x_d) + \sum_{\ell\ne k} h_\ell(0)\bigg]\,.
    \end{align*}
    It follows that for some $k\in[d]$ we have $h_k(x_1+\cdots+x_d) + \sum_{\ell\ne k} h_\ell(0)\ge a$, and hence $(x_1+\cdots x_d)e_k\in U$, as desired. 
\end{proof}

\section{Bounds from fractional hypergraph theory }
\label{sec:labelings}

In this section we recall the definition of $\rLZ_H(\delta)$ from \Cref{conj:LiZh}, along with the notions of mixed hubs and stable labelings from \cite{LiZh}.
After recalling basic definitions from fractional hypergraph theory, we introduce the refined notion of a \emph{strict} stable labeling, which will be crucial for identifying subgraphs $F\subset H$ with negligible contribution to homomorphism densities. 

\subsection{Mixed hubs and stable labelings}
\label{sec:SL}

A function 
\begin{equation}
    \label{def:xi}
    \xi:[0,1]\to[0,+\infty)\quad \text{ with }\xi(0)=1
\end{equation}
is called a \emph{compatibility function}.
Let $T$ be a finite subset of the simplex $\{(t_1,\dots, t_r): t_1,\dots,t_r\ge0, \sum_{i=1}^rt_i=1\}$ that is invariant under permutations of the $r$ coordinates. 
Following \cite{LiZh}, for such $T,\xi$, the pair $(T,\xi)$ is called a \emph{compatible collection of mixed hubs}. 
Define
\begin{equation}
    \label{def:Vol}
    \Vol(T,\xi):= \sum_{t\in T} \prod_{i=1}^r \xi(t_i)\,.
\end{equation}

\begin{defn}[Labeling]
\label{def:L}
    A mapping $\lambda:V(H)\to[0,1]$, $v\mapsto \lambda_v$ is a \emph{labeling} (for $H$) if
    \begin{enumerate}[(i)]
        \item $\sum_{v\in e} \lambda_v\in\{0,1\}$ $\forall e\in E(H)$, and
        \item $\deg_H(v)<\Delta(H)$ $\Longrightarrow$ $\lambda_v=0$.
    \end{enumerate}
    The \emph{supporting subgraph} for a labeling $\lam$ is the subgraph $F_\lam\subseteq H$ with $E(F)=\{e\in H: \sum_{v\in e}\lam_v=1\}$ and no isolated vertices.
    We also call $F_\lam$ the \emph{edge-support} of $\lam$ (to distinguish from the support of $\lam$ as a function on the vertices, as $\lam$ can be zero on some vertices of $F_\lam$).
\end{defn}

\begin{defn}[Stable labeling]
\label{def:SL}
    A labeling $\lambda$ is \emph{stable} if there is no other labeling $\lambda'$ satisfying
    \begin{enumerate}[(a)]
        \item $\sum_{v\in e} \lambda_v = \sum_{v\in e} \lambda'_v$ for all $e\in E(H)$, 
        \item $\{v: \lambda_v=0\}=\{v:\lambda'_v=0\}$, and
        \item For all $u,v\in V(H)$, we have 
        \[
        \lambda(u)=\lambda(v) \; \Longleftrightarrow \; \lambda'(u)=\lambda'(v)\,.
        \]
        In other words, the level sets of $\lambda$ and $\lambda'$ induce the same partitions of $V(H)$:
        \[
        \{\lambda^{-1}(y): y\in [0,1]\} = \{ \lambda'^{-1}(y):y\in[0,1]\}.
        \]
    \end{enumerate}
    We denote by $L_H$ the (finite) set of all stable labelings for $H$. \\    
\end{defn}

    Note that the zero labelling $\lambda\equiv 0$ always lies in $L_H$. 
    Further, 
    let 
    \begin{equation}
        \label{def:TH}
        T_H:= \big\{ (t_1,\dots, t_r) \in [0,1]^r: \{t_i\}_{i\in[r]} = \{ \lambda_v\}_{v\in e} \text{ for some $e\in E(H)$ and $\lambda\in L_H$} \big\}\,.
    \end{equation}
    That is, 
    $T_H$ is the set of all points $(t_1,\dots, t_r)\in [0,1]^r$ such that the multiset $\{t_i\}_{i\in[r]}$ is equal to the multiset $\{\lambda_v\}_{v\in e}$ of values taken by some labeling $\lambda\in L_H$ on some edge $e\in E(H)$.
    Note that 
    \[
    T_H':=T_H\setminus\{(0,\dots, 0)\}
    \]
    is a subset of the simplex, and that $T_H$ is symmetric under coordinate permutations. 

Given a compatibility function $\xi$ as in \eqref{def:xi}, we define
\begin{equation}
    \label{def:PH}
    P_H(\xi):= \sum_{\lambda\in L_H} \prod_{v\in V(H)} \xi(\lambda_v)\,.
\end{equation}
and, recalling \eqref{def:Vol},
\begin{equation}
    \label{def:VolH}
    \Vol_H(\xi):= \Vol(T_H', \xi) = \sum_{t\in T_H'} \prod_{i=1}^r \xi(t_i)\,.
\end{equation}
For $\delta>0$ define
\begin{equation}
    \label{def:rhoH}
    \rLZ_H(\delta):= \inf\big\{ \Vol_H(\xi): P_H(\xi)\ge 1+ \delta\big\}
\end{equation}
where the infimum is taken over all compatibility functions $\xi$ as in \eqref{def:xi}. 
(Note that since $L_H$ is finite, the infimum is actually taken over a bounded finite-dimensional set.)

\subsection{Fractional hypergraph theory}
\label{sec:frac}

Recall that the \emph{fractional matching number} $\nu^*(H)$ of a hypergraph $H$ is given by the following linear program:
\begin{align}
    \max & \quad \sum_{e\in E(H)} w_e \label{primal:obj}\\
    \text{subject to} & \quad
    \begin{cases}
        \sum_{e\ni v} w_e \le 1 & \forall v\in V(H),    \\
        w_e \ge0 & \forall e\in E(H)    \label{primal:con}
    \end{cases}
\end{align}
taken over $w\in \R^{E(H)}$. 
By strong linear programming duality, $\nu^*(H)$ is also given by the value for the dual linear program over $\lambda\in \R^{V(H)}$:
\begin{align}
    \min & \quad \sum_{v\in V(H)} \lambda_v \label{dual:obj}\\
    \text{subject to} & \quad
    \begin{cases}
        \sum_{v\in e} \lambda_v \ge 1 & \forall e\in E(H),\\
        \lambda_v \ge0 & \forall v\in V(H)\,. \label{dual:con}
    \end{cases}
\end{align}
In particular, for any $\lambda\in \R^{V(H)}$ satisfying \eqref{dual:con} we have
\begin{alignat}{2}
    |E(H)| &\le \sum_{e\in E(H)} \sum_{v\in e} \lambda_v && \label{ellfrac1}\\
    &= \sum_{v\in V(H)} \sum_{e\in E(H)} \lambda_v 1_{v\in e}
    &&= \sum_{v\in V(H)} \lambda_v \deg_H(v) \notag\\ %\label{ellfrac2}\\
    &  &&\le \Delta(H) \sum_{v\in V(H)} \lambda_v     \label{ellfrac3}
\end{alignat}
and hence 
\begin{equation}    \label{LB:nustar}
    |E(H)| \le \Delta(H)\nu^*(H). 
\end{equation}
Arguing similarly from \eqref{dual:obj}--\eqref{dual:con} we get
\begin{equation}
    \label{UB:nustar}
    r\nu^*(H) \le |V(H)|.
\end{equation}

Recall that a \emph{matching} for a hypergraph graph $F$ is a set of disjoint edges in $F$, and the \emph{matching number} $\nu(F)$ is the maximum size of a matching in $F$.
A \emph{transversal}\footnote{Sometimes also called a \emph{vertex cover}.} of $F$ is a set of vertices that is incident to every edge, and the \emph{transversal number} $\tau(F)$ is the minimal size of a transversal of $F$. 
Recalling the fractional matching number $\nu^*(F)$ from \eqref{primal:obj}--\eqref{primal:con}, we have
\begin{equation}    \label{nu-ineqs}
\max\{ \nu(F), |E(F)|/\Delta(F)\}\le  \nu^*(F) \le  \min\{\tau(F),|V(F)|/r\}
\,.
\end{equation}
In particular, 
\begin{equation}
    \label{frac-regular}
    H \text{ $\Delta$-regular}\quad \Longrightarrow \quad \nu^*(H) = |E(H)|/\Delta.
\end{equation}

\begin{lemma}   \label{lem:trans}
    Suppose $G$ is an $r$-graph with a transversal $S\subset V(G)$ of size $|E(G)|/\Delta(G)$. Then $\deg_G(v)=\Delta(G)$ for every $v\in S$, and $S$ is an independent set in $G$.
\end{lemma}

\begin{proof}
    Let $E_v=\{ e\in E(G): v\in e\}$. Since $S$ is a transversal we have $E(G)=\cup_{v\in S} E_v$, so
    \[
    |E(G)|\le \sum_{v\in S} |E_v| \le |S|\Delta(G).
    \]
    Since we assumed $|S|=|E(G)|/\Delta(G)$, it follows that both bounds above hold with equality. From the first equality we deduce that the sets $E_v$ are disjoint, and hence $S$ is an independent set. From the second it follows that $\deg_G(v)=\Delta(G)$ for each $v\in S$. 
\end{proof}

\subsection{Strict stable labelings}
\label{sec:SSL}

We introduce the following refinement of \Cref{def:L}.

\begin{defn}
    \label{def:SSL}
    A labeling $\lambda$ for an $r$-graph $H$ is \emph{strict} if $\sum_{v\in e} \lambda_v=1$ for all $e\in E(H)$.
\end{defn}

\begin{lemma}\label{lem:SSLsub}
Let $\lam$ be a stable labeling for $H$.
Then the restriction of $\lam$ to its supporting subgraph $F_\lam$ (see \Cref{def:L}) is a strict stable labeling for $F_\lam$.

Conversely, for any subgraph $F\subset H$ having a strict stabling labeling $\lam:V(F)\to [0,1]$, the extension $\ol\lam$ to $V(H)$ defined by setting $\lam_v=0$ for all $v\in V(H)\setminus V(F)$ is a stable labeling for $H$.
\end{lemma}

\begin{proof}
    The first claim is immediate from the definitions.
    For the second, we only need to check that $\sum_{v\in e}\lam_v=0$ for all $e\in E(H)\setminus E(F)$. This is immediate if $e\subset V(H)\setminus V(F)$. On the other hand, if $e\in E(H)\setminus E(F)$ and $v\in e\cap V(F)$, then $\deg_F(v)<\Delta(H)$ and hence $\lam_v=0$ by definition of a labeling. The claim follows.
\end{proof}

\begin{lemma}
    \label{lem:noSSL}
    Let $F$ be an $r$-graph with $\Delta(F)\ge2$, and let $f$ be a symmetric measurable function on $\cX^r$ satisfying \eqref{assu:K} with $K\lessapprox1$ and \eqref{assu:kappa}.
    If $F$ does not have a strict stable labeling, then $t(F,f/p)=O_F( p^{\eps_0})$ for some $\eps_0=\eps_0(F)>0$ depending only on $F$. 
\end{lemma}

\begin{remark}
    Note that the lemma holds vacuously for the case that $\Delta(F)=1$, since in this case $F$ has a strict stable labeling given by taking $\lambda_v=1/r$ when $\deg_F(v)=1$ and $\lambda_v=0$ otherwise. 
\end{remark}

\Cref{lem:noSSL} will follow from a generalization of \cite[Lemma 6.4]{BGLZ} together with the following:

\begin{lemma}
    \label{lem:SSL-frac}
    An $r$-graph $F$ has a strict stable labeling if and only if $ |E(F)|=\Delta(F)\nu^*(F) $.
\end{lemma}

\begin{proof}
Suppose $F$ has a strict stable labeling $\lambda:V(F)\to[0,1]$. Then $\lambda$ satisfies \eqref{dual:con} and hence \eqref{LB:nustar}. %\eqref{ellfrac1}--\eqref{ellfrac3}.
Since $\sum_{v\in e}\lambda_v\le1$ for all $e\in E(F)$ we have equality in \eqref{ellfrac1}, and since $\lambda_v=0$ for any $v$ of degree less than $\Delta(F)$ we also have equality in \eqref{ellfrac3}. 
Thus \eqref{LB:nustar} holds with equality.

Now suppose $|E(F)|=\Delta(F)\nu^*(F)$.
Let $\lambda\in\R^{V(F)}$ be a vertex labeling as in \eqref{dual:con} with value $\sum_{v\in V(F)} \lambda_v = \nu^*(F)$. Since equality holds in \eqref{LB:nustar} we must have equality in \eqref{ellfrac1} and \eqref{ellfrac3}, so $\sum_{v\in e}\lambda_v=1$ for all $e\in E(F)$ and $\lambda_v=0$ whenever $\deg_F(v)<\Delta(F)$. 
To obtain a strict stable labeling it only remains to verify the stability conditions (b) and (c) in \Cref{def:SL}. 
We can do this by adding linear constraints of the form $\lambda_u=0$ or $\lambda_u=\lambda_v$ for $u\ne v$ until $\lambda$ is uniquely determined. This claim follows.
\end{proof}

\begin{proof}[Proof of \Cref{lem:noSSL}]
    Let $F$ be an $r$-graph with no strict stable labeling. From \Cref{lem:SSL-frac} it follows that $\nu^*(F)>|E(F)|/\Delta(F)$.
    Then the constant weight vector $w\equiv 1/\Delta(F)\in [0,\frac12]$ is not a local maximizer for the problem \eqref{primal:obj}--\eqref{primal:con}, and hence we can perturb this weighted vector to obtain $\t w: E(F)\to[0,\frac23]$ such that $\t w$ still lies in the feasible region specified in \eqref{primal:con}, and $\sum_{e\in E(F)} \t w_e \ge (|E(F)|+c)/\Delta(F)$ for some $c\in (0, \Delta(F)\nu^*(F)-|E(F)|)$ depending only on $F$. 
    By Lemmas \ref{lem:Finner} and \ref{lem:L1-bound},
    \begin{align*}
        t(F,f) 
        \le \prod_{e\in E(F)} \bigg( \int_{\cX^r} f^{1/\t w_e} d\mu^r \bigg)^{\t w_e}
%        &\le \bigg( \int_{\cX^r} fd\mu^r \bigg)^{\sum_{e\in E(F)} \t w_e}\\
%        &\le \bigg( \int_{\cX^r} f d\mu^r \bigg)^{(|E(F)|+\tau)/\Delta(F)}\\
        \le \|f\|_{L^1}^{\sum_{e\in E(F)} \t w_e}
        &\le \|f\|_{L^1}^{(|E(F)|+c)/\Delta(F)}
        \lessapprox_F p^{|E(F)|+c} 
    \end{align*}
    and the claim follows with $\eps_0=c/2$, say. 
\end{proof}

\begin{defn}[Critical subgraph]
    \label{def:critical}
    A subgraph $F\subseteq H$ is called \emph{critical} for $H$ if $\nu^*(F)= |E(F)|/\Delta(H)$; or equivalently, if $\Delta(F)=\Delta(H)$ and $F$ has a strict stable labeling.
    We write 
%    \begin{equation}        \label{def:Fcrit}
    \[
        \Fcrit:= \{ F\subset H: \nu^*(F) = |E(F)|/\Delta(H)\}
    \]
%    \end{equation}
    for the set of all critical subgraphs for $H$, excluding $H$ itself. 
\end{defn}

From multilinearity of $t(H,\cdot)$, \Cref{lem:smallDelta} and \Cref{lem:noSSL}, for any $r$-graph $H$ and any symmetric measurable function $f$ on $\cX^r$ satisfying \eqref{assu:K} with $K\lessapprox1$ and \eqref{assu:kappa}, we have
\begin{align}
    t(H, 1+ f/p)   
    &= \sum_{F\subseteq H} t(F, f/p) 
    = 1+ t(H, f/p) + \sum_{F\in \Fcrit} t(F, f/p) + o(1) \,.\label{sum:Fcrit}
\end{align}

\section{A general class of $r$-graphs}
\label{sec:gen}

For this section we fix a nonempty $r$-graph $H$ of maximum degree $\Delta:=\Delta(H)\ge2$. 

Recall that the loose path $P_\ell^{(r)}$ of length $\ell$ has $\ell$ edges and $\ell(r-1)+1$ vertices, with $\ell-1$ vertices $w_1,\dots, w_{\ell-1}$ having degree $2$, all other vertices having degree 1, and each edge containing 1 or 2 of the $w_i$. 

\begin{defn}[(Very) good subgraphs]    \label{def:good}
We say a subgraph $F\subseteq H$ is \emph{good} if 
$F$ has a unique transversal $S$ of minimal size $k:=\tau(F)=|E(F)|/\Delta$.

We say $F$ is \emph{very good} if $F$ is good and
the following additional properties hold:
        \begin{enumerate}[(VG1)]
        \item\label{VG1} 
        For any $v\in V(F)\setminus S$, 
        there are $k$ disjoint edges in $F$ that do not contain $v$.
        \item\label{VG2} 

        For any $v\in S$ there is a subgraph $F'\subset F$ with $k+1$ edges such that
        \begin{itemize}
            \item $F'$ is a vertex-disjoint union of loose paths,
            \item $S\subset V(F')$,
            \item $\deg_{F'}(v)=2$, and
            \item for any $w\ne v$, if $\deg_{F'}(w)=2$ then $w\notin  S$.
        \end{itemize}
    \end{enumerate}
\end{defn}

See \Cref{fig:goodcover} for an illustration of property (VG\ref{VG2}).

\begin{figure}
    \centering
    \includegraphics[width=0.6\linewidth]{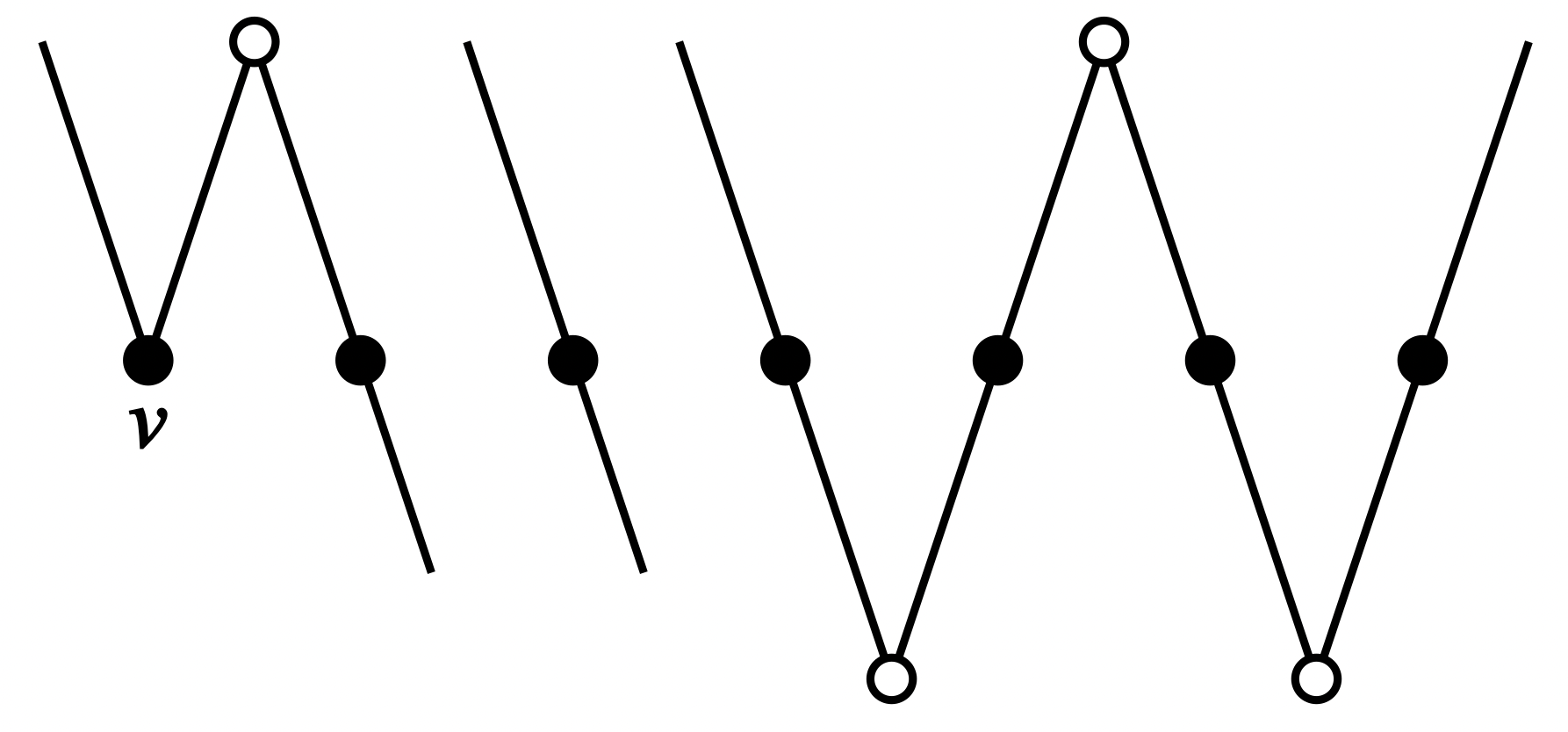}
    \caption{Illustration of a subgraph $F'\subset F$ as in \Cref{def:good}(VG\ref{VG2}). Vertices are circles and edges are straight lines. Vertices in $S$ are solid black. Vertex $v\in S$ has degree 2 in $F'$, while all other degree-2 vertices in $F'$ (open circles) lie outside $S$. (Other vertices of $F'$ are not depicted.) In this example $F'$ consists of 3 loose paths of lengths 3, 1 and 5.}
    \label{fig:goodcover}
\end{figure}

\begin{assumption}
\label{assu:good}
For any critical 
proper subgraph $F\subset H$, 
$F$ is very good.
\end{assumption}

The main result of this section is the following.

\begin{prop}
    \label{prop:gen}
    Let $(\cX,\mu)$ be a probability space and suppose the $r$-graph $H$ satisfies \Cref{assu:good}.
    Then for any fixed $\delta>0$,
    \begin{equation}    \label{gen:LB1}
        \frac{\phi_H^{(\cX,\mu)}(p,\delta) }{p^\Delta\log(1/p)} \ge \rnew_H(\delta) + o(1)
    \end{equation}
    where
    \begin{equation}
        \label{def:rnew}
        \rnew_H(\delta):=
        \begin{cases}
         \min\big\{ \delta^{\Delta/|E(H)|},\, r\beta_H(\delta)\big\} 
         & \nu^*(H)=|E(H)|/\Delta\\
         r\beta_H(\delta) & \nu^*(H)>|E(H)|/\Delta
        \end{cases}
    \end{equation}
    where $\beta_H(\delta)$ was defined in \eqref{def:betaH}.
    Moreover, if $H$ itself is very good, then
    \begin{equation}    \label{gen:LB2}
        \frac{\phi_H^{(\cX,\mu)}(p,\delta) }{p^\Delta\log(1/p)} \ge r\beta_H(\delta)+o(1).
    \end{equation}
\end{prop}

\begin{remark}
    A careful tracking of errors in the proof gives a convergence rate $o(1)$ in \eqref{gen:LB1}, \eqref{gen:LB2} that is uniform for $\delta\le (\log\log(1/p))^c$ for a small constant $c=c(H)>0$). 
\end{remark}

We showed in \eqref{sum:Fcrit} that the dominant contribution to homomorphism counts $t(H,1+f/p)$ comes from subgraphs in the class $\Fcrit$ from \Cref{def:critical}.
Now we show that under \Cref{assu:good}, the class $\Fcrit$ reduces to a simpler class $\Fstar$ that we now define.
Recall that we write $\cI_H\subset 2^{V(H)}$ for the set of independent sets in an $r$-graph $H$. Recall also the notation $V^\star(H),H^\star$ from \eqref{def:Hstar}.
    
\begin{defn}
    \label{def:SIS}
    For an $r$-graph $H$ and
    $S\in\cI_H$, let $F_S$ denote the subgraph of $H$ with $V(F_S)=S\cup N_H(S)$ and $E(F_S)=\{e\in E(H): e\cap S\ne \emptyset\}$.
    Let
    \begin{align}
        \cIspan& := \{ S\in \cI_{H^\star} : F_S=H\}\,,
    \end{align}
    which may be empty.
    Let 
    \begin{equation}    \label{def:cFH'}
        \Fstar:=\{F_S:S\in\cI_{H^\star}\setminus\cIspan\}=\{ F_S: S\in\cI_{H^\star} \}\setminus \{H\}.
    \end{equation}
\end{defn}

Note that $\Fstar\subseteq\Fcrit$. Indeed,  $S\in\cI_{H^\star}$ is a transversal for $F_S$, so from \eqref{nu-ineqs} we have 
\[
\nu^*(F_S)\le \tau(F_S) \le |S| = |E(F_S)|/\Delta(H)\le \nu^*(F_S).
\]

\begin{lemma}   \label{lem:good}
Suppose %$H$ satisfies \Cref{assu:good}. 
every critical proper subgraph of $H$ is good. 
Then the following hold:
%\quad\\
    \begin{enumerate}[(a)]
        \item  %If $F\subset H$ is good then $F$ is critical. 
        $\Fcrit=\Fstar$.  
        \item If $\tau(H)=|E(H)|/\Delta$ and $S_1,S_2\subset V(H)$ are two transversals of size $\tau(H)$ then $S_1\cap S_2=\emptyset$. 
    \end{enumerate}
\end{lemma}

\begin{proof}
    For (a), let  $F\in\Fcrit$. Then $\tau(F)=|E(F)|/\Delta \le |E(F)|/\Delta(F)$.
    But from \eqref{nu-ineqs} it follows that $\tau(F)=|E(F)|/\Delta(F)$ and $\Delta(F)=\Delta$. The claim now follows from \Cref{lem:trans}. 

    For (b),  suppose $S_1\cap S_2\ne \emptyset$. 
    We claim that $S_1\setminus S_2$ and $S_2\setminus S_1$ are both minimal transversals for the proper subgraph $F=F_{S_1\setminus S_2}$, contradicting the assumption that $F$ is good. Indeed, $S_1\setminus S_2$ is a transversal for $F$ by definition. Since $S_2$ is a transversal for $H$, any edge $e\in E(F)$ must intersect $S_2$. If $e\cap S_1\cap S_2\ne\emptyset$ then $S_1$ is not an independent set since $e$ also intersects $S_1\setminus S_2$, which contradicts part (a). Hence, any edge of $F$ must intersect $S_2\setminus S_1$, so $S_2\setminus S_1$ is a transversal as claimed.
\end{proof}

\subsection{The case of $2$-graphs}

Before proceeding with the proof of \Cref{prop:gen}, we verify that the assumptions cover all connected $2$-graphs.

\begin{prop}    \label{prop:2graph}
    If $H$ is a connected 2-graph of maximum degree $\Delta\ge2$, then $H$ satisfies \Cref{assu:good}. Moreover, if $H$ is irregular and critical then $H$ is very good. 
\end{prop}

Together with \Cref{prop:gen} and \eqref{frac-regular}, the above gives the sharp bound established in \cite{BGLZ}:
\[
\Phi_H(n,p,\delta) \ge (1+o(1))\frac1{r!}\rbi_H(\delta) n^rp^{\Delta(H)}\log(1/p)
\]
for $n=\omega(1)$, where $\rbi_H$ was defined in \eqref{def:rbi}. A matching upper bound for $\omega(1)\le n\le o(p^{-\Delta(H)})$ follows from computing $I_p(Q)$ for weighted graphs $Q$ such that $Q-p$ is supported on cliques or anticliques of appropriate sizes; see \cite[Section 3]{BGLZ}.

The proof of \Cref{prop:2graph} below, together with the specialization of the proof of \Cref{prop:gen} to the case $r=2$, are not the same as the proof from \cite{BGLZ}, as many extensions of the ideas from that work are needed to accommodate $r\ge3$. However, we make crucial use of several fundamental ideas from \cite{BGLZ}. 

\begin{proof}
    Fix an arbitrary critical proper subgraph $F\subset H$. 
    Then $\nu^*(F) = |E(F)|/\Delta$ and $\Delta(F)=\Delta$.
    Let $F_1,\dots, F_s$ be the connected components of $F$.
    Since
    \[
    \frac1\Delta \sum_{i=1}^s |E(F_i)| \le \sum_{i=1}^s \frac{|E(F_i)|}{\Delta(F_i)} \le \sum_{i=1}^s \nu^*(F_i) = \nu^*(F) = \frac{|E(F)|}\Delta = \frac1\Delta \sum_{i=1}^s |E(F_i)|
    \]
    we have that $F_i$ is critical and $\Delta(F_i)=\Delta$ for each $i$.
    Moreover, for each $i\in [s]$, since $F_i$ is connected and $\Delta(F_i)=\Delta$, it follows that $F_i$ is irregular, as otherwise $F_i$ would be a connected component of $H$, whereas we assumed $H$ is connected and $F$ is a proper subgraph.

    By \cite[Lemma 6.3]{BGLZ}, for each $i\in [s]$ we have $\tau(F_i) = |E(F_i)|/\Delta$, and by remarks under Lemma 6.1 in \cite{BGLZ} we have that $F_i$ has a unique minimal transversal $S_i$.  
    Thus $F_i$ is good. 
    Then $S:= S_1\cup\cdots S_s$ is the unique transversal of $F$ of size $k:=\tau(F)=\tau(F_1)+\cdots+\tau(F_s)$.
    Again by remarks under Lemma 6.1 in \cite{BGLZ} we have that $S$ is the unique minimal transversal for $F$, and hence $F$ is good. 

    Now we verify that $H$ has properties (VG\ref{VG1}), (VG\ref{VG2}). For this we use the following:

    \begin{claim}
        \label{claim:2graph}
        For any $v\in V(F)\setminus S$, there is a matching $M\subseteq F$ with $k$ edges such that $v\notin V(M)$. 
    \end{claim}

    \begin{proof}
        Without loss of generality assume $v\in V(F_1)\setminus S_1$. Since $F_1$ is connected and irregular, with $\deg_{F_1}(u) = \Delta$ for all $u\in S_1$, the neighborhood in $F_1$ of any set $U\subseteq S$  has size at least $|U|+1$. 
        Hence, with $F_1':= F_1[V(F_1)\setminus\{u\}]$, the neighborhood in $F_1'$ of any set $U\subseteq S$ has size at least $|U|$.
        By Hall's matching theorem, $F_1'$ has a matching $M_1$ of size $\tau(F_1)$. Combining with arbitrary matchings $M_i$ for $F_i$ of size $\tau(F_i)$ for each $2\le i\le s$ we obtain the desired matching $M$.
    \end{proof}

    From \Cref{claim:2graph} it follows that $F$ has property (VG\ref{VG1}). 
    For property (VG\ref{VG2}), let $v\in S$ and let $u$ be an arbitrary neighbor of $v$. Then $u\in V(F)\setminus S$. Applying \Cref{claim:2graph} with $u$ in place of $v$ we get a matching $M$ of $k$ edges that does not contain the edge $\{u,v\}$. Then taking $F'$ to be the union of $M$ and the edges $\{u,v\}$, the required properties of $F'$ are easily verified. 

    If $H$ is irregular and critical, then we can argue exactly as we did above for the irregular connected components $F_i$ to see that $H$ is very good. 
\end{proof}

\subsection{Proof of \Cref{prop:gen}}

For this and the remaining subsections of this section we fix a probability space $(\cX,\mu)$ and assume $H$ satisfies \Cref{assu:good}.
To establish \Cref{prop:gen} we need the following five lemmas.
For these lemmas we fix a symmetric measurable function $f$ on $\cX^r$, assuming
\begin{equation}
    \label{assu:f}
    \text{ $f$ satisfies \eqref{assu:kappa} and $J_p(f)\lessapprox p^\Delta$}.
\end{equation}
(In \Cref{lem:gen-adapt} we assume a slightly stronger bound on $J_p(f)$.)
We emphasize that implicit constants and rates in the asymptotic notation hold uniformly for $f$ subject to the stated hypotheses.

\begin{lemma}
    \label{lem:gen-low}
    Let $F\subseteq H$ such that $k:=\tau(F)=|E(F)|/\Delta$. 
    If $v\in V(F)$ is such that there are $e_1,\dots, e_k\in E(F)$ with $\{v\},e_1,\dots, e_k$ disjoint, then for any $T\ge1$,
    \begin{equation}
        t(F, f/p\,;\, v\to D_{\ge Tp^\Delta})
        \lessapprox_H T^{-1}.
    \end{equation}
    In particular, for any fixed $\eps>0$,
    \[
    t(F, f/p\,;\, v\to D_{\ge p^{\Delta-\eps}}) = o(1)\,.
    \]
\end{lemma}

\begin{proof}
    Writing $b:=T p^{\Delta}$, we have
    \begin{align*}
        t(F,f\,;\, v\to D_{\ge b}) 
        &= \int_{\cX^{V(F)}} 1_{x_v\in D_{\ge b}} \prod_{e\in E(F)} f(x_e) d\mu_{V(F)}(x)\\
        &\le \int_{\cX^{V(F)}} 1_{x_v\in D_{\ge b}} \prod_{i=1}^k f(x_{e_i}) d\mu_{V(F)}(x)\\
        &= \mu(D_{\ge b}) \|f\|_{L^1}^k\\
        &\le b^{-1}\|f\|_{L^1}^{k+1}\\
        &\lessapprox_H b^{-1} p^{\Delta(k+1)}
        = b^{-1}p^{|E(F)|+ \Delta}
        = T^{-1}p^{|E(F)|}
    \end{align*}
    where we applied \eqref{D-Markov} in the penultimate line and \Cref{lem:L1-bound} in the final line.
    The claim follows.
\end{proof}

% \begin{lemma}
%    \label{lem:gen-high0}
%    Let $F\subset H$ be as in \Cref{assu:good}(VG\ref{def:good.a},\ref{VG2}).
%    For any $v\in S$ and any fixed $\eps>0$,
%    \[
%    t(F, f/p\,;\, v\to D_{<p^\eps}) = o(1)\,.
%    \]
% \end{lemma}

% \begin{proof}
%     Fix $v\in S$ and let $e_0,\dots, e_k\in E(F)$ be as in \Cref{assu:good}(VG\ref{VG2}).
%     We have
%     \begin{align*}
%         t(F, f\,;\, v\to D_{<p^\eps})
%         &= \int_{\cX^{V(F)}} 1_{x_v\in D_{<p^\eps}} \prod_{e\in E(F)} f(x_e) d\mu_{V(F)}(x)\\
%         &\le \int_{\cX^{V(F)}} 1_{x_v\in D_{<p^\eps}} \prod_{i=0}^k f(x_{e_i}) d\mu_{V(F)}(x)\\
%         &= \|f\|_{L^1}^{k-1} \int_{\cX^{e_0\cup e_1}}1_{x_v\in D_{<p^\eps}} f(x_{e_0})f(x_{e_1})d\mu_{e_0\cup e_1}(x_{e_0\cup e_1}) \\
%         &=\|f\|_{L^1}^{k-1} \int_{D_{<p^\eps}} d_f(x_v)^2 d\mu(x_v)\\
%         &<p^\eps \|f\|_{L^1}^{k-1} \int_{D_{<p^\eps}} d_f(x_v) d\mu(x_v)\\
%         &= p^\eps \|f\|_{L^1}^k\\
%         &\lessapprox p^{k\Delta + \eps} = p^{|E(F)|+\eps}
%     \end{align*}
%     where in the final line we applied \Cref{lem:Lp-bounds}.
%     The claim follows. 
% \end{proof}

%\nickQ{Previous version of the following under previous stronger \Cref{assu:good}(VG\ref{VG2})  is under comment}

\begin{lemma}
    \label{lem:gen-high}
    Let $F\subseteq H$ be %as in \Cref{def:good}(VG\ref{def:good.a},\ref{VG2}).
    very good.
    For any $v\in S$ and $b>0$,
    \begin{equation}
        t(F, f/p\,;\, v\to D_{<b}) \lessapprox_H  b^{1/2k}.
    \end{equation}
    In particular, for any fixed $\eps>0$,
    \[
    t(F, f/p\,;\, v\to D_{<p^\eps}) = o(1)\,.
    \]
\end{lemma}

\begin{proof}
%    From \Cref{assu:good}(VG\ref{VG2}), there are disjoint subgraphs $F_1,\dots, F_q$ of $H$ such that the edge sets $E(F_1),\dots,E(F_q)$ are a partition of $\{e_0,\dots,e_k\}$, and for each $i\in[q]$, $F_i$ is either a loose path or a single edge. 
    Fix $v\in S$ and let $F'\subset F$ be as in (VG\ref{VG2}). Let $F_1,\dots, F_q$ be the connected components $F'$. 
    Assume the $F_i$ are labeled so that $\ell_1,\dots, \ell_q$ is non-increasing, and let $q'\in[q]$ be the largest label with $\ell_{q'}\ge2$. 
    Label the degree-2 vertices of $F_i$ as $W_i=\{w_{i,j}\}_{j=1}^{\ell_i-1}$ and let $W:=W_1\cup\cdots\cup W_{q'}$. Thus, we have $v=w_{i_*,j_*}$ for some $i_*\in[q']$ and $j_*\in [\ell_1-1]$, and $W\setminus\{v\} \subseteq V(F)\setminus S$. 

    We may assume $b\le 1$ as otherwise the claim is trivial.
    Set $b':= b^{-1/2k}p^{\Delta}$. 
    From (VG\ref{VG1}) and \Cref{lem:gen-low}, 
    \[
    t(F, f/p\,;\, w_{i,j}\to D_{\ge b'}) \lessapprox_H b^{1/2k}
    \]
    for all $(i,j)\ne (i_*,j_*)$. Since $|E(F)|=k\Delta$,   it now suffices to show
    \begin{equation}    \label{goal:gen-high}
        t(F,f\,;\, v\to D_{<b}, W\setminus\{v\}\to D_{<b'}) \lessapprox_H b^{1/2k} p^{k\Delta}.
    \end{equation}
    We can bound the left hand side above by
    \begin{align*}
        &t(F_1\cup\cdots\cup F_q, f\,;\, v\to D_{<b}, W\setminus\{v\}\to D_{<b'})\\
        &\quad =\|f\|_{L^1}^{q-q'} t(F_{i_*}, f\,;\, v\to D_{<b}, W_{i_*}\setminus\{v\}\to D_{<b'})
        \prod_{i\in[q']\setminus \{i_*\}} t(F_i, f\,;\, W_i\to D_{<b'}).
    \end{align*}
    From \Cref{lem:path},
    \[  
    t(F_{i_*}, f\,;\, v\to D_{<b}, W_{i_*}\setminus\{v\}\to D_{<b'})
    < b (b')^{\ell_1-2}\|f\|_{L^1}
    \]
    (note this covers the case that $W_1\setminus\{v\}=\emptyset$, with $\ell_1=2$)
    and
    \[
    t(F_i, f\,;\, W_i\to D_{<b'}) < (b')^{\ell_i-1}\|f\|_{L^1}\,,\quad i\in[q']\setminus\{i_*\}
    \]
    so with $L:=\sum_{i=1}^q (\ell_i-1)=\sum_{i=1}^{q'}(\ell_i-1)$, the left hand side in \eqref{goal:gen-high} is at most 
    \begin{align*}
        & b (b')^{L-1} \|f\|_{L^1}^q
        \lessapprox_H b (b')^{L-1}  p^{\Delta q} 
        = b^{1-\frac{L-1}{2k}} p^{ \Delta (q + L-1)}.
    \end{align*}
    Since $|E(F_i)|=\ell_i$ for each $i\in[q]$ we have
    \[
    k+1 = \sum_{i=1}^q |E(F_i)| = q+L
    \]
    so
    \begin{align*}
        &t(F,f\,;\, v\to D_{<b}, W\setminus\{v\}\to D_{<b'})
        \lessapprox_H b^{1-\frac{L-1}{2k}} p^{k\Delta } 
        \le b^{1/2}p^{k\Delta}.
    \end{align*}
    This gives \eqref{goal:gen-high} to complete the proof.
\end{proof}

The next three lemmas will be proved in subsequent subsections.

\begin{lemma}
    \label{lem:gen-H<}
    Suppose $\tau(H)>|E(H)|/\Delta= \nu^*(H)$.
    There exists $\eps_1=\eps_1(H)>0$ depending only on $H$ such that
    for any $b\ge p^{\eps_1}$,
    \[
    t(H,f/p) = t(H, 1_{D_{<b}^r}f/p) + o(1).
    \]
\end{lemma}

\begin{lemma}
    \label{lem:gen-vout}
    Suppose $\tau(H)=|E(H)|/\Delta$. If $v\in V(H)$ does not lie in any transversal of $H$ of size $\tau(H)$, then for any $b\ge p^{\eps_1}$, with $\eps_1=\eps_1(H)$ as in \Cref{lem:gen-H<}, 
    \[
    t(H, f/p\,; \, v\to D_{\ge b}) = o(1).
    \]
\end{lemma}

\begin{lemma}
    \label{lem:gen-adapt}
    Assume $J_p(f)\le (\log\log(1/p))^{1/10}p^\Delta$.
    Suppose $\tau(H)=|E(H)|/\Delta$.
    There exists $p_0=p_0(H)>0$ such that for any $p\in (0,p_0]$ there exists $b=b_{f,H}$ with 
    \[
    \exp( - \log^{3/4}(1/p)) \le b\le \exp( - \log^{1/2}(1/p))
    \]
    (in particular $p^{o(1)}\le b\le o(1)$)
    such that for any transversal $U$ for $H$ of size $\tau(H)$ and any $v,w\in U$,
    \[
    t(H,f/p\,;\, v\to D_{\ge b}, w\to D_{<b}) %= o(1). 
    \le (\log\log(1/p))^{-1/4} = o(1).
    \]
\end{lemma}

\begin{proof}[Proof of \Cref{prop:gen}]
From \Cref{lem:phi-tphi} and the continuity of $\rnew_H$ and $\be_H$, it suffices to show that for any fixed $\delta>0$ and symmetric measurable $f$ on $\cX^r$ satisfying \eqref{assu:kappa}, if $t(H,1+f/p)\ge 1+\delta$
then
\begin{equation}
    \label{gen-goal1}
    p^{-\Delta}J_p(f) \ge \rnew_H(\delta) + o(1)
\end{equation}
and furthermore,
\begin{equation}
    \label{gen-goal2}
    p^{-\Delta}J_p(f) \ge r\be_H(\delta)+o(1)
\end{equation}
if $H$ is very good.
Fix such a function $f$ and $\delta>0$. We may assume 
\begin{equation}    \label{gen:assuJ}
    J_p(f)\le (\log\log(1/p))^{1/10}p^\Delta
\end{equation}
as we are done otherwise. Thus, $f$ meets the hypotheses of Lemmas \ref{lem:gen-low}--\ref{lem:gen-adapt}.

Letting $b=b_{f,H}$ be as in \Cref{lem:gen-adapt}, we set
\begin{align}
    \al_f&:=  p^{-1} \bigg( \int_{D_{<b}^r} f^\Delta d\mu^r\bigg)^{1/\Delta}
    =\|1_{D_{<b}^r}f/p \|_{L^\Delta}\\
%    \,,\qquad
    \be_f &:=  p^{-\Delta} \int_{D_{\ge b}\times D_{<b}^{r-1}} f^\Delta d\mu^r 
    =\|1_{D_{\ge b}\times D_{<b}^{r-1}}f/p \|_{L^\Delta}^\Delta
    \,.
\end{align}
From \eqref{sum:Fcrit} and \Cref{lem:good}, 
\begin{align}
    t(H, 1+ f/p)   
%    &= \sum_{F\subseteq H} t(F, f/p) \notag\\
    &= 1+ t(H, f/p) + \sum_{F\in \Fcrit} t(F, f/p) + o(1)    \notag\\
    &= 1+ t(H, f/p) + \sum_{F\in \Fstar} t(F, f/p) + o(1)\,.\label{gen-1}
\end{align}

Now consider an arbitrary very good subgraph $F=F_S\in \Fstar\cup\{H\}$ (thus we consider $F=H$ only if $H$ itself is very good).
From \Cref{assu:good} and Lemmas \ref{lem:gen-low} and \ref{lem:gen-high}, for any fixed $\eps>0$ sufficiently small we have
\begin{align}
    t(F, f/p) 
    & = t(F, f/p\,;\, S\to D_{\ge b}, V(F)\setminus S\to D_{<p^{\Delta-\eps}}) + o(1)  \notag\\
    &\le t(F, f/p\,;\, S\to D_{\ge b}, V(F)\setminus S\to D_{<b}) + o(1)\,. \label{gen-2}
\end{align}
Since each edge of $F$ contains exactly one vertex in $S$, we have
\begin{align}
    &t(F, f\,;\, S\to D_{\ge b}, V(F)\setminus S\to D_{<b}) \notag\\
    &\quad= \int_{\cX^{V(F)}} \prod_{v\in S} 1_{x_v\in D_{\ge b}}\prod_{w\in V(F)\setminus S} 1_{x_w\in D_{<b}}\prod_{e\in E(F)} f(x_e) d\mu_{V(F)}(x)  \notag\\
    &\quad\le \|f1_{D_{\ge b}\times D_{<b}^{r-1}}\|_{L^\Delta}^{|E(F)|} 
    = \|f1_{D_{\ge b}\times D_{<b}^{r-1}}\|_{L^\Delta}^{\Delta |S|}
    = \beta_f^{|S|} \label{gen-Fbeta}
\end{align}
where we applied \Cref{cor:Finner}. 
Together with \eqref{gen-1}, \eqref{gen-2} and recalling the definition of $\Fstar$, we have
\begin{equation}
    t(H, 1+ f/p) \le 1+  \sum_{\emptyset\ne S\in\cI_{H^\star} } \beta_f^{|S|}+o(1) 
    = i_{H^\star}(\be_f)+o(1)\label{gen-case0}
\end{equation}
if $H$ is very good, 
and otherwise
\begin{equation}
    t(H, 1+ f/p) \le 1+ t(H,f/p) + \sum_{\emptyset\ne S\in\cI_{H^\star} \setminus \cIspan} \beta_f^{|S|}+o(1)   \label{gen-3}
\end{equation}

For the contribution of $t(H,f/p)$ in \eqref{gen-3} we split into three cases:\\

\emph{Case 1: $\nu^*(H)> |E(H)|/\Delta$}. 
In this case $t(H,f/p) = o(1)$ by \Cref{lem:noSSL}.
Moreover, $\cIspan=\emptyset$ as otherwise we would have $\tau(H)=|E(H)|/\Delta$, whereas $\nu^*(H)\le \tau(H)$ (see \eqref{nu-ineqs}).
Thus 
\begin{equation}    \label{gen-case1}
    t(H, 1+ f/p) \le 1+ \sum_{\emptyset\ne S\in\cI_{H^\star} } \beta_f^{|S|}+o(1)
    = i_{H^\star}(\be_f) +o(1).
\end{equation}

\emph{Case2: $\nu^*(H)=|E(H)|/\Delta<\tau(H)$.}
Then from \Cref{lem:gen-H<} (since $b=p^{o(1)}\ge p^{\eps_1}$) and \Cref{cor:Finner} we have
\[
    t(H,f/p) = t(H, 1_{D_{<b}^r} f/p) + o(1) \le \al_f^{|E(H)|} +o(1).
\]
Moreover, $\cIspan=\emptyset$ as in Case 1.
Thus
\begin{equation}    \label{gen-case2}
    t(H, 1+ f/p) \le 1+ \al_f^{|E(H)|}+\sum_{\emptyset\ne S\in\cI_{H^\star} } \beta_f^{|S|}+o(1)
    = \al_f^{|E(H)|} + i_{H^\star}(\beta_f) + o(1).
\end{equation}

\emph{Case 3: $k:=\nu^*(H)=|E(H)|/\Delta = \tau(H)$.}
Let $U_1,\dots, U_{r'}\subset V(H)$ be the transversals of $H$ of size $k$. Since $\tau(H)=k$ we have $r'\ge1$. 
From \Cref{lem:good} we know that $U_i\subseteq V^\star(H)$ for each $i$
and $U_1,\dots, U_{r'}$ are pairwise disjoint. 
Writing $W:= V(H)\setminus(U_1\cup\cdots\cup U_{r'})$, $D_+:=D_{\ge b}$ and $D_-:=D_{<b}$, from Lemmas \ref{lem:gen-vout} and \ref{lem:gen-adapt} we have
\begin{equation}    \label{gen-inout}
    t(H, f/p) = \sum t(H, f/p \,;\, U_1\to D_\pm, \dots, U_{r'}\to D_\pm, W\to D_-) +o(1)
\end{equation}
where the sum runs over the $2^{r'}$ assignments of signs to sets $D_\pm$. 
That is, the dominant contribution to $t(H,f/p)$ comes from mappings where for each of the transversals $U_1,\dots, U_{r'}$, all of $U_i$ is sent to either the high-degree set $D_+$ or the low-degree set $D_-$, and all remaining vertices are all sent to $D_-$.

Now we argue that \eqref{gen-inout} can be further refined to the contribution where at most one of $U_1,\dots, U_{r'}$ is sent to $D_+$. (We only need to argue this if $r'\ge2$.)
Indeed, applying \Cref{lem:edge-weight} with $\theta_1=\theta_2=2/3$, we have
\[  
t(H, f/p \,;\, U_1\cup U_2\to D_+) 
\le t(H, f/p \,;\, U_1\cup U_2\to D_{\ge p^{\Delta/3}} =o(1)\,.
\]
By symmetry the same holds with any pair $\{i,j\}\subseteq[r']$ in place of $\{1,2\}$. 
This shows that we may eliminate any assignment of signs on the right hand side of \eqref{gen-inout} with more than one +. 
Hence,
\[
    t(H,f/p) = t(H, f/p \,;\, V(H)\to D_-) 
    + \sum_{i=1}^{r'} t(H, f/p \,;\, U_i\to  D_+, V(H)\setminus V_i \to D_-) + o(1).
\]

By the same lines as in \eqref{gen-Fbeta}, for any $i\in[r']$ we have
\[
t(H, f/p \,;\, U_i\to  D_+, V(H)\setminus U_i \to D_-)
\le \be_f^{k}
\]
and hence
\begin{align}
    t(H,f/p) &\le  t(H, f/p \,;\, V(H)\to D_-) 
    + r'\be_f^k + o(1)  \notag\\
    &= t(H, 1_{D_-^r}f/p) + r'\be_f^k + o(1) \notag\\
    &\le \al_f^{|E(H)|} + r'\be_f^k + o(1)
\end{align}
where in the last line we applied \Cref{cor:Finner}.
Combining with \eqref{gen-3} we have
\begin{align}
    t(H, 1+ f/p)
    &\le 1+ \al_f^{|E(H)|} + r'\be_f^k + \sum_{\emptyset\ne S\in\cI_{H^\star} \setminus \cIspan} \beta_f^{|S|}+o(1)  \notag\\
    &=\al_f^{|E(H)|} + i_{H^\star}(\be_f) +o(1)
    \label{gen-case3}
\end{align}
where in the second line we noted that $r'\be_f^k = \sum_{S\in \cIspan} \be_f^{|S|}$.

We now conclude the proof of \Cref{prop:gen}. 
Consider first the case that $H$ is very good or $\nu^*(H)>|E(H)|/\Delta$. 
From \eqref{gen-case0} and \eqref{gen-case1}, and our assumption that $t(H,1+f/p)\ge 1+\delta$, we get that
\[
    1+\delta \le  i_{H^\star}(\beta_f) + o(1).
\]
On the other hand, from \eqref{Jp-quad},
\begin{align*}
    J_p(f)
    \ge  \int_{\cX^r} f^\Delta d\mu^{r} 
    \ge (\al_f^\Delta + r\beta_f)p^\Delta
\end{align*}
for all $p$ sufficiently small. 
Hence, for some $\delta'=\delta+o(1)$, 
\begin{align*}
    p^{-\Delta}J_p(f)
    &\ge \inf_{\beta\ge0}\{ r\beta  : i_{H^\star}(\be_f) \ge 1+\delta'\} \\
    &= r\beta_H(\delta') = r\beta_H(\delta)+o(1)
\end{align*}
with the last equality due to the continuity of $\beta_H$. 
This concludes the proof for the case that $H$ is very good or $\nu^*(H)>|E(H)|/\Delta$.

Now if $\nu^*(H)=|E(H)|/\Delta$, from \eqref{gen-case2} and \eqref{gen-case3} and our assumption that $t(H,1+f/p)\ge 1+\delta$, we get that
\[
1+\delta \le \al_f^{|E(H)|} + i_{H^\star}(\beta_f) + o(1).
\]    
Hence, for some $\delta'=\delta+o(1)$, 
\begin{align*}
    p^{-\Delta}J_p(f)
    &\ge \inf_{\al,\beta\ge0}\{ \al^\Delta + r\beta  :\al^{|E(H)|} + i_{H^\star}(\be_f) \ge 1+\delta'\} \\
    &= \min\{ \delta'^{\Delta/|E(H)|}, r\beta_H(\delta')\}\\
    &= \min\{ \delta^{\Delta/|E(H)|}, r\beta_H(\delta)\} +o(1)
\end{align*}
as claimed, where in the second line we used \Cref{lem:cvx}, noting that $\al\mapsto \al^{|E(H)|/\Delta}$ and $i_{H^\star}(\cdot)$ are convex, and in the final line we used the continuity of $x\mapsto x^{\Delta/|E(H)|}$ and $\beta_H$ to replace $\delta'$ with $\delta$.
\end{proof}

\subsection{Proof of \Cref{lem:gen-H<}}

    Let $\eps_1>0$ to be taken sufficiently small depending on $H$.
    Fixing an arbitrary vertex $v\in V(H)$ and setting $b=p^{\eps_1}$, it suffices to show
    \begin{equation}
        \label{H<goal1}
        t(H, f\,;\, v\to D_{\ge b}) = o(p^{|E(H)|}).
    \end{equation}
    From \Cref{lem:low-to-low} we may assume $\deg_H(v)=\Delta$.
    Let $F:=H[V(H)\setminus\{v\}]$ be the induced subgraph of $H$ on all vertices but $v$. Thus, $F$ is obtained from $H$ by removing $v$ and the $\Delta$ edges incident to it. 
    We have
    \begin{equation}    \label{H<.1}
        t(H, f\,;\, v\to D_{\ge b}) \le \mu(D_{\ge b}) t(F, f)\le b^{-1}\|f\|_{L^1} t(F,f)\lessapprox p^{\Delta-\eps_1}t(F,f)
    \end{equation}
    where we used \Cref{lem:L1-bound}.
    We establish \eqref{H<goal1} in three cases:\\

    \emph{Case 1: $\Delta(F)<\Delta$.}
    Then from \Cref{lem:smallDelta},
    \[
    t(F, f) \lessapprox_H p^{|E(F)|\Delta/\Delta(F)} = p^{(|E(H)|-\Delta)\Delta/\Delta(F)}
    \]
    which together with \eqref{H<.1} gives
    \[
    t(H, f\,;\, v\to D_{\ge b}) 
    \lessapprox_H 
    p^{\Delta-\eps_1 + (|E(H)|-\Delta)\Delta/\Delta(F) }.
    \]
    Note that $H$ is not a star as otherwise we would have $\tau(H)=1=|E(H)|/\Delta$. Hence $|E(H)|>\Delta$, and
    \[
    \Delta+ (|E(H)|-\Delta)\frac\Delta{\Delta(F)} 
    = |E(H)| + \Big( \frac{\Delta}{\Delta(F)}-1\Big) (|E(H)|-\Delta)>|E(H)|.
    \]
    We can thus take $\eps_1>0$ sufficiently small that
    \[
    t(H, f\,;\, v\to D_{\ge b}) \lessapprox_H p^{|E(H)| + \eps_1/2}
    \]
    giving \eqref{H<goal1} in this case.\\

    \emph{Case 2: $\Delta(F)=\Delta$ and $F\notin \Fcrit$.}
    Then from \Cref{lem:noSSL} we have 
    \[
    t(F, f) \ls_H p^{|E(F)|+\eps_0(F)} = p^{|E(H)|-\Delta + \eps_0(F)}
    \]
    which together with \eqref{H<.1} and taking $\eps_1\le \eps_0(F)/2$ yields \eqref{H<goal1}.\\

    \emph{Case 3: $\Delta(F)=\Delta$ and $F\in \Fcrit$.} We claim that this case is incompatible with the assumption that $|E(H)|/\Delta<\tau(H)$ and is hence empty. Indeed, 
    from \Cref{assu:good}, $F$ has a transversal $S\subset V^\star(F)$ of size 
    \[
    \tau(F) = \frac{|E(F)|}\Delta = \frac{|E(H)|}\Delta -1.
    \]
    We note that $S$ must be disjoint from the neighborhood $N_H(v)$ of $v$ in $H$. Indeed, if an element $w\in S$ were adjacent to $v$, then $w$ would have degree at most $\Delta-1$ in $F$ since $F$ contains none of the edges incident to $v$. From \Cref{lem:trans} we know all elements of $S$ have degree $\Delta$. 
    Then since $S$ is disjoint from $N_H(v)$, it follows that $U:= S\cup \{v\}$ is a transversal for $H$, and hence
    \[
    \tau(H)\le |S|+1 = \frac{|E(H)|}\Delta
    \]
    which contradicts our assumption $\tau(H)>|E(H)|/\Delta$.
    Hence Case 3 is empty, so we have established \eqref{H<goal1}.
    \qed

\subsection{Proof of \Cref{lem:gen-vout}}

The argument is identical to the argument establishing \eqref{H<goal1} in the proof of \Cref{lem:gen-H<}, except for the final lines of the argument for Case 3 that $\Delta(F)=\Delta$ and $F\in\Fcrit$. 
By the same argument we find that $S\cup\{v\}$ is a transversal of $H$, where $S$ is the unique transversal of $F$ of minimal size. 
But this contradicts our assumption that $v$ does not lie in any transversal of $H$. This completes the proof. \qed

\subsection{Proof of \Cref{lem:gen-adapt}}

    Let $p_0=p_0(H)>0$ to be taken sufficiently small, and $p\in (0,p_0]$. 
    For integer $h\ge0$ set
    \[
    b_h:= \exp( - (5k)^h\log^{1/2}(1/p))\,,\qquad 
    B_h:= D_{\ge b_h}\,.
    \]
    Set $k:=\tau(H)$ and $M:=\lceil (\log\log(1/p))^{1/2}\rceil$. 
    
    First we claim that for any transversal $U$ of size $k$ and any $u,v\in U$ and $b'\ge b''>0$, 
    \begin{equation}\label{gen-adapt1}
        t(H,f/p \,;\,u\to D_{\ge b'}, v\to D_{<b''}) \lessapprox_H (b'')^{1/2k}/b'\,.
    \end{equation}
    Indeed, fixing such $U,u,v,b',b''$ and letting $F:=H[V(H)\setminus \{u\}]$, we have $F=F_{S'}$ with $S'=S\setminus \{u\}$ (recall the notation from \Cref{def:SIS}).
From \Cref{lem:gen-high} we have
\begin{align*}
    t(H, f\,;\, u\to D_{\ge b'}, v\to D_{<b''})
    &\le \mu(D_{\ge b'}) t( F,f\,;\, v\to D_{<b''})\\
    &\lessapprox_H (b')^{-1} \|f\|_{L^1} (b'')^{1/2k} p^{|E(F)|}\\
    &\lessapprox_H (b')^{-1} (b'')^{1/2k}p^{|E(F)|+\Delta}
    = (b')^{-1} (b'')^{1/2k} p^{|E(H)|}
\end{align*}
    where in the second line we applied \eqref{D-Markov} and \Cref{lem:gen-high} and in the final line we used \Cref{lem:L1-bound}. \eqref{gen-adapt1} now follows from homogeneity. 

    Since $b_{h+1} = b_h^{5k}$, we have
    \[
    \frac{b_{h+1}^{1/2k}}{b_h}  = b_h^{3/2} \le b_0^{3/2} = \exp( - \tfrac32\log^{1/2}(1/p)).
    \]
    Applying \eqref{gen-adapt1} with $b'=b_h, b''=b_{h+1}$, we thus have
    \begin{align*}
        t(H, f/p \,;\,u\to B_h, v\to \cX\setminus B_{h+1}) 
        &\le  \exp( O_H(\log\log(1/p)) - \tfrac32\log^{1/2}(1/p)) 
    \end{align*}
    and hence
    \begin{equation}    \label{adapt-gen2}
        t(H, f/p \,;\, u\to B_h, v\to \cX\setminus B_{h+1}) \le 
        \exp\big( - \log^{1/2}(1/p))\big)
        \qquad \forall u,v\in U,\;  h\in [M]
    \end{equation}
    taking $p_0$ sufficiently small. 
    For $u,v\in V(H), h\in [M]$ set 
    \[
    t_h(u,v) := t(H, f/p \,;\,u\to B_h, v\to \cX\setminus B_h)
    \]
    put
    \[
    \delta_0:= (\log\log(1/p))^{-1/4},
    \]
    and suppose toward a contradiction that
    \begin{equation}    \label{adapt-gen-X}
        \forall h\in [M] \; \exists \text{ transversal }U \text{ and } u,v\in U \text{ such that } t_h(u,v) \ge \delta_0\,.
    \end{equation}
    From the pigeonhole principle it follows there exists a transversal $U_0$, vertices $u_0,v_0\in U_0$, and $I\subset[M]$ of size $|I|\gs_H M$ such that $t_h(u_0,v_0)\ge \delta_0$ for all $h\in I$. 
    With \eqref{adapt-gen2} we get that for any $h\in I$, 
    \begin{align*}
        t_h' 
        &:= t(H, f/p \,;\, u_0\to B_h, v_0\to B_{h+1}\setminus B_h)\\
        &=t_h(u_0,v_0) - t(H, f/p \,;\, u_0\to B_h, v_0\to \cX\setminus B_{h+1})\\
        &\ge \delta_0 - \exp( -\log^{1/2}(1/p))\\
        &\ge \delta_0/2. 
    \end{align*}
    On the other hand, since the sets $B_h\times (B_{h+1}\setminus B_h), h\ge1$ are disjoint in $\cX^2$, we have
    \[
    (\log\log(1/p))^{1/4}\le 
    \delta_0 M\ls_H 
    \sum_{h\in I} t_h' \le t(H, f/p) \le K \le (\log\log(1/p))^{1/10},
    \]
    which is a contradiction if $p_0$ is sufficiently small. 
    Hence \eqref{adapt-gen-X} does not hold, so there exists $h_0\in [M]$ such that $t_h(u,v)<\delta_0$ for every transversal $U$ and $u,v\in U$. 
    The claim now follows with 
    \[
    b = b_{h_0} \ge \exp( - (5k)^M\log^{1/2}(1/p)) 
    \ge \exp( - \log^{3/4}(1/p))\,.
    \]
    \qed

\section{Complete $r$-partite $r$-graphs}
\label{sec:rpartite}

For this section we let $H$ be a copy of $K^{(r)}_{m_1,\dots, m_r}$, with $V(H)$ the disjoint union of parts $V_i$ of size $m_i$. We assume $m_1\le m_2\le \cdots\le m_r$. 
Write $\Delta:=\Delta(H)= m_2m_3\cdots m_r$.
For convenience of notation we identify 
\begin{equation}    \label{rpartite-V}
    V_i = \{i\}\times [m_i]\,.
\end{equation}
For $A\subseteq [r]$ we abbreviate
\begin{equation}
    V_A:= \bigcup_{i\in A} V_i .
\end{equation}
Note that when $H$ is the regular complete $r$-partite $r$-graph $K^{(r)}_{m,\dots, m}$ we have $V(H)=[r]\times [m]$ and $V_A = A\times [m]$. 

\Cref{thm:main}(\ref{main.r-partite}) covers the cases $H$ is regular, i.e.\ assuming
\begin{equation}    \label{A1}
    \tag{A1}
    m_1=\cdots= m_r=m \ge 2
\end{equation}
and the case that $H$ has a single part of minimal size, i.e.
\begin{equation}\label{A2}
    \tag{A2}
    1\le m_1<m_2\le \cdots\le m_r\,.
\end{equation}

In \Cref{sec:rpartite.UB} we establish the upper bound for cases \eqref{A1}, \eqref{A2}.

In \Cref{sec:rpartite.LB} we show the lower bound holds more generally: we only need to assume 
\begin{equation}    \label{A3}
    \tag{A3}
    m_2>1.
\end{equation}

\subsection{Proof of \Cref{thm:main}(\ref{main.r-partite}), upper bound}
\label{sec:rpartite.UB}

In this subsection we establish the upper bound
\begin{equation}    \label{rpartite-UB}
    R_H(n,p,\delta) \le (1+o(1)) \rbi_H(\delta) n^rp^\Delta\log(1/p)
\end{equation}
for $n=\omega(1)$, under the assumptions of \Cref{thm:main}(\ref{main.r-partite}) (i.e.\ assuming \eqref{A1} or  \eqref{A2}).
We also show that in this case we have $\rLZ_H=\rbi_H$. 

We begin with the easier case \eqref{A2}. 
First we verify the right hand side of \eqref{rpartite':rH} is equal to $\rbi_H$. 
Indeed, under \eqref{A2} we have $V^\star=V_1$ and $H^\star$ has no edges, so 
\[
i_{H^\star}(x) = \sum_{S\subseteq V_1} x^{|S|} = (1+x)^{m_1}
\]
and hence
\[
\beta_H(\delta) = (1+ \delta)^{1/m_1} - 1.
\]
Since $H$ is irregular in this case we have 
\[
\sigma_H(\delta)=r\beta_H(\delta) = r[(1+ \delta)^{1/m_1} - 1]
\]
giving the right hand side of \eqref{rpartite':rH} as desired.

Now we compute $\rLZ_H$ under \eqref{A2}. Since stable labelings are supported on the vertices of max degree, and the vertices of max degree form the independent set $V_1$, the stable labelings for $H$ are exactly the indicators $1_S$ of subsets $S\subseteq V_1$. Thus,
\[
P_H(\xi) = i_{H^\star}(\xi(1)) =(1+\xi(1))^{m_1}.
\]
Since each edge contains one vertex of $V_1$, we have that $T_H'$ is the set of $r$ standard basis vectors of $\R^r$, so
\[
\Vol_H(\xi) = r\xi(1). 
\]
Thus, writing $\be=\xi(1)$, we have
\[
\rho_H(\delta) = \inf\{ r\be: (1+\be)^{m_1}\ge 1+\delta) \} = r[(1+\delta)^{1/m_1}-1] = \sigma_H(\delta)
\]
as desired.
\eqref{rpartite-UB} now follows from \Cref{prop:upper}.

Turning to the case \eqref{A1} that $H$ is regular,  the independence polynomial is 
\begin{align*}
    i_H(x) = 1+ \sum_{k=1}^m r{m\choose k} x^k = 1+ r\big[ (1+x)^m-1].
\end{align*}
Indeed, $S\subset V(H)$ is independent if and only if it is a subset of one of the  $r$ parts $V_1,\dots, V_m$, so for $1\le k\le m$ there are $r{m\choose k}$ independent sets of size $k$.
Solving the equation $i_H(\beta)= 1+\delta$ shows that
\begin{equation}
    \beta_H(\delta) = \big( 1+\tfrac\delta r)^{1/m}-1
\end{equation}
($H$ is regular so $H^\star=H$)
so the right hand side of \eqref{rpartite:rH} is equal to $\rbi_H(\delta)$.

From \Cref{prop:upper} it remains to show 
\begin{align}   \label{rpartite:rhoH}
    \rLZ_H(\delta) 
    &= 
%   \begin{cases}
%        r\big[(1+\frac\delta{r'})^{1/m}-1\big] & r'<r\\
        \min\Big\{r \big[ \big( 1+ \frac\delta{r}\big)^{1/m} -1\big]\,,\; 
    \delta^{1/m} \Big\} %& 2\le r'= r \,.
%    \end{cases}
    \,.
\end{align}
For this we have the following lemma. 
Recall the notation \eqref{rpartite-V}.

\begin{lemma}   \label{lem:rpartite-SL}
    Assume $H=K^{(r)}_{m_1,\dots, m_r}$ for general $1\le m_1\le \cdots\le m_r$. Let $r'\le r$ be the largest integer such that $m_1=m_{r'}$.
    $H$ has two types of stable labelings:
    \begin{itemize}
        \item 
        For some $i\in[r']$ and a proper nonempty subset $S\subset V_i$ we have $\lambda_v=1_{v\in S}$.
        \item For some nonempty $A\subseteq[r']$ we have $\lambda_v=1/|A|$ for all $v\in V_A$ and $\lambda_v=0$ for all $v\notin V_A$. 
    \end{itemize}
\end{lemma}

\begin{proof}
    Let $\lambda$ be a stable labeling for $H$.
    By definition we have $\lambda_v=0$ for all $v\notin V^\star=\cup_{i=1}^{r'} V_i = [r']\times[m_1]$. 

    Suppose $\lambda_v=1$ for some $v\in V^\star$. Without loss of generality suppose $v\in V_1$. 
    Then $\lambda_u=0$ for any $u\notin V_1$, since $H$ contains edges containing both $u$ and $v$. 
    It then follows that $\lambda_w\in\{0,1\}$ for any other $w\in V_1$. 

    If $\lambda_v=\theta \in (0,1)$ for some $v\in V_i$, with $i\in [r']$,
    then we claim $\lambda_w=\theta $ for all $w\in V_i$. 
    Indeed, letting $e$ be an edge containing $v$, we must have $\sum_{u\in e\setminus \{v\}} \lambda_u = 1-\theta $. Then with $e':= (e\setminus \{v\})\cup \{w\}$ it follows that $\sum_{u\in e'}\lambda_u=1$ and so $\lambda_w=\theta $.

    Hence, if $\lambda$ takes some value in $(0,1)$, then there are numbers $\theta _1,\dots, \theta _{r'}\in [0,1)$, not all zero, such that $\lambda_v= \theta _i$ for all $v\in V_i$. 
    Let $A=\{ i\in[r']: \theta _i>0\}$. Now for $\lambda$ to be stable we see that the numbers $\theta _i, i\in A$ must all be the same number $\theta \in(0,1)$.
    In order to have $\sum_{v\in e}\lambda_v\in\{0,1\}$ for every $e\in E(H)$ we must have $\theta =1/|A|$. This completes the proof. 
\end{proof}

As a consequence, when $H$ is regular, for a compatibility function $\xi:[0,1]\to[0,+\infty)$ we have
\begin{align}
    P_H(\xi) 
    &=1+ r\sum_{k=1}^{m}{m\choose k} \xi(1)^k 
    + \sum_{k=2}^{r} {r\choose k} \xi(\tfrac1k)^{km} \notag\\
    &= 1+ r[ (1+\xi(1))^{m}-1] +  \sum_{k=2}^{r} {r\choose k} \xi(\tfrac1k)^{km}\,. \label{rpartite-PH}
\end{align}
Furthermore, it also follows from \Cref{lem:rpartite-SL} that
\[
T_H'= 
\big\{ \tfrac1{|A|} 1_A :  A\subseteq [r], 1\le |A|\le r\big\}
\]
so 
\begin{equation}
    \Vol_H(\xi) = \sum_{k=1}^{r} {r\choose k}\xi(\tfrac1k)^{k}\,.
\end{equation}
Setting $\al_k:=\xi(\tfrac1k)^k$, we have
\begin{align}
    \rLZ_H(\delta)
    &= \inf\bigg\{ \sum_{k=1}^{r} {r\choose k} \al_k 
    : r[ (1+\al_1)^{m}-1] +  \sum_{k=2}^{r} {r\choose k} \al_k^{m}
    \ge \delta 
    \bigg\}.
\end{align}
Since the functions $\al\mapsto r[(1+\al)^{m}-1]$ and $\al\mapsto {r\choose k}\al^{m}$ are convex,
from \Cref{lem:cvx} we obtain
\[
    \rLZ_H(\delta) 
    = \min\Big\{ r\Big[ \Big(1+ \frac\delta {r}\Big)^{1/m}-1\Big] \,,\,
    C(m, r) \delta^{1/m} \Big\}
\]
where
\[
    C(m,r) = \min_{2\le k\le r} {r\choose k}^{1-1/m}=1\,.
\]
We thus obtain \eqref{rpartite:rhoH} to complete the proof of \eqref{rpartite-UB} under \eqref{A1}.

\subsection{Proof of \Cref{thm:main}(\ref{main.r-partite}), lower bound}
\label{sec:rpartite.LB}

Under the assumptions of \Cref{thm:main}(\ref{main.r-partite}) we have that \eqref{A1} of \eqref{A2} holds. Then
since $\Delta(H)= m_2\cdots m_r$, from \Cref{thm:CDP}, \eqref{Delta'UB}, \eqref{LB:Phi-phi}, \Cref{lem:phi-tphi} and \Cref{prop:gen} it only remains to verify 
that $H$ satisfies \Cref{assu:good}. (We already showed in the proof of the upper bound that $\rLZ_H=\rbi_H$, and since $H$ is regular we have $\rnew_H=\rbi_H$ by \eqref{frac-regular}.) We do this under the more general assumption \eqref{A3}:

\begin{claim}
    Under \eqref{A3}, \Cref{assu:good} holds for $H$.
\end{claim}

\begin{figure}[t]
    \centering
    \includegraphics[width=0.5\linewidth]{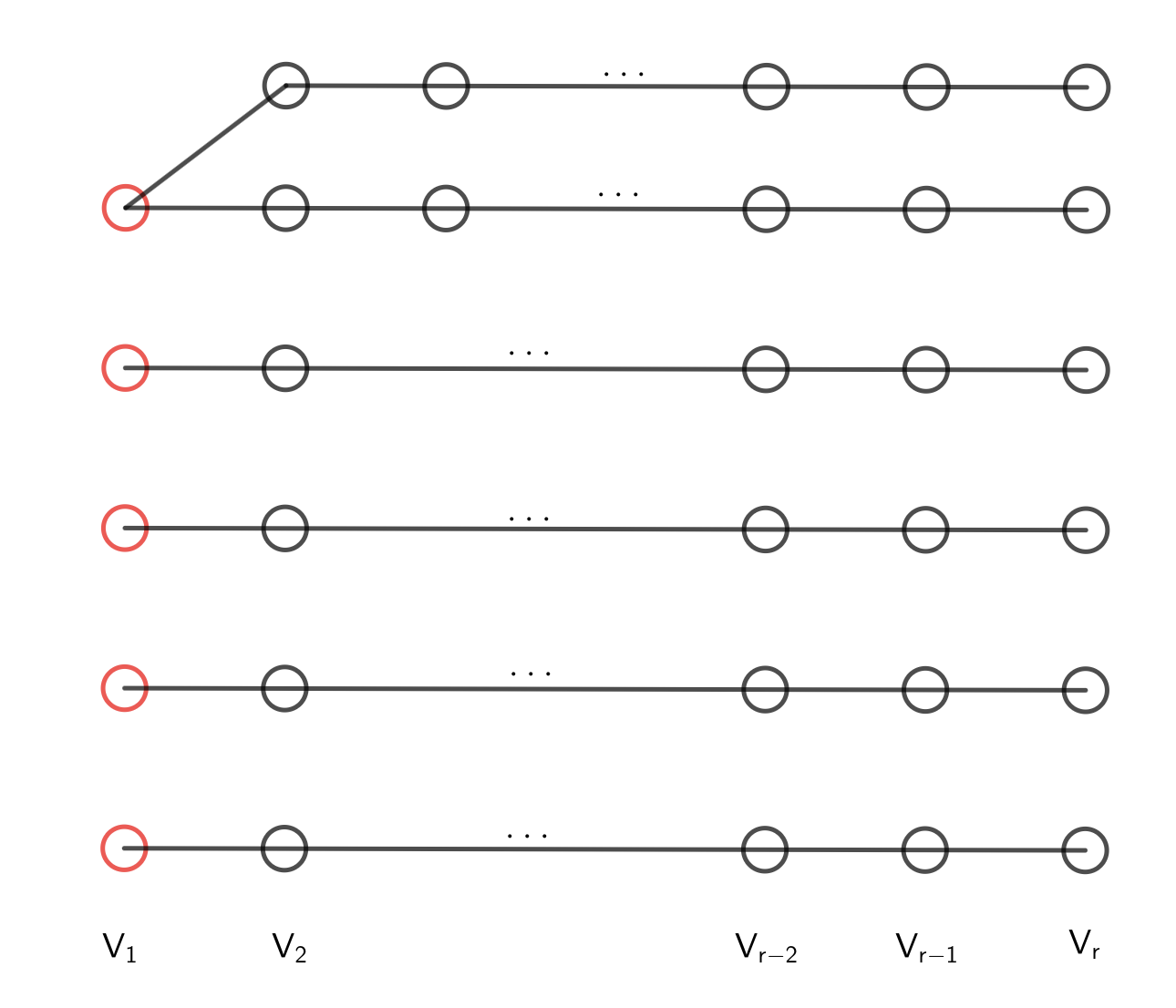}
    \caption{
    To verify (VG\ref{VG2}) of \Cref{def:good}, we select the edges of $F'$ as depicted. In this example $H$ has parts of size 6, and $F$ is the subgraph consisting of all edges incident to the set $S$ of 5 red vertices in $V_1$. The lines depict the choice of $|S|+1=6$ edges for $F'$.
} 
    \label{fig:rpartite-good}
\end{figure}

\begin{proof}
    Let $F$ be a proper subgraph of $H$ with $\nu^*(F)=|E(F)|/\Delta$. 
    From \Cref{lem:SSLsub}, $F$ is the supporting subgraph of some stable labeling $\lam$ for $H$. From \Cref{lem:rpartite-SL} we have $\lam=1_S$ for some $S\subset V_i$ and $i\in[r']$, as the second type of labeling in \Cref{lem:rpartite-SL} has edge-support $H$. 
    Since $\lam$ is strict for $F$ it follows that $S$ is a transversal for $F$. Thus $|S|=\tau(F) = |E(F)|/\Delta=\nu^*(F)$.
    Since $H$ is a complete $r$-partite graph it follows that $F$ is the induced subgraph of $H$ on all vertices except $V_i\setminus S$, and $S$ is the unique transversal of $F$ of minimal size. This shows that $F$ is good. 

    By symmetry we may henceforth assume $S=\{1\}\times [k]\subset V_1$. 

    Turning to establish (VG\ref{VG1}), let $v\in V(F)\setminus S = V_{[2,r]}$. By symmetry we may assume $v=(2,m)$. 
    Since $k<m$ we may take $e_i=[r]\times \{i\}$ for each $1\le i\le k$ to verify (VG\ref{VG1}).

    Now we establish (VG\ref{VG2}). 
    Letting $v\in S$, by symmetry we may take $v=(1,k)$. Then taking $e_i=[r]\times\{i\}$ for $1\le i\le k$ as above, along with $e_0=\{(k,1), (2,k+1),\dots, (m,k+1)\}$ (see \Cref{fig:rpartite-good}) verifies (VG\ref{VG2}). Indeed, $v$ lies in a loose path of length 2 consisting of $e_0,e_1$, while all other elements of $S$ lie in loose paths of length 1. Any vertex in $V_{[2,r]}$ lies in at most one of the edges $e_0,\dots, e_k$. The claim follows.
\end{proof}

\section{Cycles}
\label{sec:cycles}

In this section we prove \Cref{thm:main}(\ref{main.cycle}).
Throughout the section we write $H$ for $C_\ell^{(r)}$ and identify $V(H)$ with the congruence classes mod $\ell$, which we label by $[\ell]$. 
Thus addition of vertices is taken mod $\ell$. 
We denote
\begin{equation}
    d:= \text{gcd}(\ell, r).
\end{equation}
For the proof of \eqref{cycle:rHsub} for proper subgraphs of $H$, we write $H'$ for an arbitrary proper subgraph of $H$ on the same vertex set, with $\Delta(H')=r$.

\subsection{Proof of \Cref{thm:main}(\ref{main.cycle}), upper bound}

In this subsection we show
\begin{equation}    \label{cycle-UB}
    R_H(n,p,\delta) \le (1+o(1)) \frac1{r!}\min\{\delta^{r/\ell}, r\be_H(\delta)\} n^rp^\Delta\log(1/p)
\end{equation}
for $n=\omega(1)$.
One verifies that the same lines show
\begin{equation}    \label{cycle-UB'}
    R_{H'}(n,p,\delta) \le (1+o(1)) \frac1{(r-1)!} \be_H(\delta) n^rp^\Delta\log(1/p)\,.
\end{equation}
From \Cref{prop:upper}, for \eqref{cycle-UB} we only need to verify that
\begin{equation}    \label{cycle:rLZ}
    \rLZ_H(\delta) = \min\{\delta^{r/\ell}, r\be_H(\delta)\} \qquad \forall \delta\ge0.
\end{equation}
For this we have the following lemma.

\begin{lemma}
\label{lem:cycles-SL}
    $H$ has two types of stable labelings $\lam$:
    \begin{enumerate}[$(i)$]
        \item $\lam_v= 1_{v\in S}$ for some  independent set $S\subset[\ell]$.
        \item For some set $U_0\subseteq [d]$ we have $\lam_v=\frac{d}{r|U_0|}1_U$
        where $U=\cup_{u\in U_0} \{ v\in [\ell]: d|(v-u)\}$ is the union of congruence classes mod $d$ that intersect $U_0$. 
    \end{enumerate}
\end{lemma}

Note that type $(ii)$ includes the constant labeling $\lam_v\equiv \frac1r$, with $U_0=[d]$, and the constant labeling is the only labeling of type $(ii)$ when $r$ and $\ell$ are coprime.
Types $(i)$ and $(ii)$ are disjoint except when $r|\ell$, i.e. $d=r$, 
in which case their intersection is the set of labelings of type $(ii)$ with $|U_0|=1$, i.e.\ $\lam=1_S$ for $S$ one of the congruence classes mod $r$.  

\begin{proof}
    Let $\lam$ be a stable labeling of $H$. The class $(i)$ of labelings is clearly stable, so it remains to show that $\lam$ falls in class $(ii)$ if there exist two vertices $u, v$ belonging to the same edge such that $\lam(u), \lam(v) > 0$. Assume there are such $u,v$. 
    
    First we show that $\lam$ must be a strict stable labeling. 
    Without loss of generality, assume $u = 1$ and $1 < v \le r$. We have $\lam(1) + \lam(v) \le 1$, so $\lam(v) < 1$. This means that among $\lam(v+1), \ldots, \lam(v+r-1)$, there must be a non-zero label (recall we take addition mod $\ell$). It follows that any two consecutive vertices with nonzero labels are within distance $r$ on the cycle, and any edge of $H$ must contain a non-zero labeled vertex.
    Hence, the total sum of labels of each edge must be $1$, so $\lam$ is strict.

    Now since $\lam$ is strict, it is easy to see that for any $u$, $\lam(u) = \lam(u+d)$. Hence, the labels of $1, \ldots, d$ determine a unique strict stable labeling of $H$. Note that since $\lam(1) + \ldots + \lam(r) = \frac rd \left(\lam(1) + \ldots + \lam(d) \right) = 1$, we have $\lam(1) + \ldots + \lam(d) = \frac dr$. 
    
    Now we claim that any two non-zero labels in $\lam(1), \ldots, \lam(d)$ must be the same. Without loss of generality, assuming again $\lam(1), \lam(v) > 0$ for some $1 < v \le d$, it suffices to show $\lam(1) = \lam(v)$. If $\lam(1) \ne \lam(v)$, then construct another labeling $\lam'$ such that $\lam'(1 + dk) = \frac{\lam(1) + \lam(v)}{2}, \lam'(v + dk) = \frac{\lam(1) + \lam(v)}{2}$ for all $k \ge 0$, where indices are taken modulo $m$, and $\lam'(t) = \lam(t)$ for all other indices $t$. One checks that $\lam'$ is a valid labeling where all edges have labels sum to $1$, and $\lam'(t) = 0$ if and only if $\lam(t) = 0$. Since $\lam(1) \ne \lam(v)$, $\lam' \not \equiv \lam$, then $\lam$ cannot be a stable labeling by definition. 
    Thus $\lam_v=c1_{U_0}$ for some $U_0\subset [d]$ and some constant $c>0$. For the labels in an edge to sum to 1 we must have $c=d/r|U_0|$. The claim follows.
\end{proof}

As a consequence, for a compatibility function $\xi:[0,1]\to[0,+\infty)$ we have
\begin{align}
    P_H(\xi) 
    &= i_H(\xi(1)) + \sum_{k=1+ 1_{r|\ell}}^d {d\choose k}\xi(\tfrac d{kr})^{\ell k/d} \,.\label{cycle-PH}
\end{align}
Furthermore, from \Cref{lem:cycles-SL} it follows that
\[
T_H'= \{e_1,\dots,e_r\}\cup \bigcup_{k=1}^d 
\bigg\{ \tfrac1{|U|} 1_U :  U\in {[r]\choose rk/d} \bigg\}
\]
so 
\begin{equation}
    \Vol_H(\xi) = r\xi(1) + \sum_{k=1+1_{r|\ell}}^d {r\choose rk/d}\xi(\tfrac d{rk})^{rk/d} .
\end{equation}
Letting $\al_\lam:=\xi(\lam)^{1/\lam}$, we have
\begin{align}
    \rLZ_H(\delta)
    &= \inf\bigg\{ r\al_1+  \sum_{k=1+1_{r|\ell}}^{d} {r\choose rk/d} \al_{d/rk}
    : i_H(\al_1) + \sum_{k=1+1_{r|\ell}}^d {d\choose k} \al_{d/rk}^{\ell/r} \ge 1+\delta\bigg\}
\end{align}
Since the functions $i_H$ and $\al\mapsto \al^{\ell/r}$ are convex,
from \Cref{lem:cvx} we obtain
%\begin{equation}    \label{rpartite-case1}
%    \rLZ_H(\delta) 
%    = r[(1+\delta)^{1/m}-1]
%\end{equation}
%when $r'=1$, and
%\begin{equation}    \label{rpartite-case2}
\begin{align*}
    \rLZ_H(\delta) &=
    \min\big\{ r\be_H(\delta) \,,\; C'(\ell, r)\delta^{r/\ell}\big\}
\end{align*}
%\end{equation}
%when $2\le r'\le r$,
where
%\begin{equation}    \label{def:Cmrr'}
\[
    C'(\ell,r) = \min_{1+1_{r|\ell}\le k\le d} {r\choose rk/d} {d\choose k}^{-r/\ell} \,.
\]
With $k=d$ we see 
\[
C'(\ell,r) \le 1.
\]
%\end{equation}
On the other hand, writing $\ell=id$ and $r=jd$, we note that
\[
{r\choose rk/d} = {jd\choose jk} \ge {d\choose k}^j = {d\choose k}^{r/d}
\]
and hence 
\[
C'(\ell, r) \ge \min_{1+1_{r|\ell}\le k\le d} {d\choose k}^{r(\frac1d-\frac1\ell)}  \ge {d\choose d}^{r(\frac1d-\frac1\ell)}= 1.
\]
Thus $C'(\ell, r)=1$. 
We thus obtain \eqref{cycle:rLZ}, which gives  to complete the proof of \eqref{cycle-UB}.

\subsection{Proof of \Cref{thm:main}(\ref{main.cycle}), lower bound}

Recall $\rnew_H$ from \eqref{def:rnew}. 
Since $H$ is regular and has a strict stable labeling (the constant labeling $\lam_v\equiv\frac1r$), from \Cref{lem:SSL-frac} it follows that $\rnew_H(\delta)= \rbi_H(\delta)=\min\{\delta^{r/\ell}, r\be_H(\delta)\}$. 
Since $\Delta(H)= r$, from \Cref{thm:CDP}, \eqref{Delta'UB}, \eqref{LB:Phi-phi}, \Cref{lem:phi-tphi} and \Cref{prop:gen}, to prove \eqref{cycle:rH} it only remains to establish the following:

\begin{lemma}
    \Cref{assu:good} holds for $H$.
\end{lemma}

For \eqref{cycle:rHsub}, as $H'$ is a proper subgraph of $H$, the proof below shows that either $\nu^*(H')>|E(H')|/r$ or $H'$ is very good, and the claim similarly follows from \Cref{prop:gen}. 

\begin{proof}
    Let $F$ be an arbitrary proper subgraph of $H$ with $\nu^*(F)=|E(F)|/r$. Thus $\Delta(F)=\Delta(H)=r$ and $F$ has a strict stable labeling $\lam$. 
    From \Cref{lem:SSLsub}, $\lam$ extends to a stable labeling for $H$, which is not strict for $H$ since $F$ is a proper subgraph. From \Cref{lem:cycles-SL} it follows that $\lam_v=1_{v\in S}$ for some independent set $S$. Since $\lam$ is strict for $F$ it follows $S$ is a transversal for $F$ of the minimal size $k:=\tau(F) = |E(F)|/r$, with $\deg_F(u)=r$ for all $u\in S$. 
    Thus, $E(F)$ is the disjoint union $E(F)=\sqcup_{u\in S} E_u$ where $E_u$ is the set of $r$ edges incident to $u$. 
    
    The proof will be complete once we verify that $F$ satisfies the three criteria of \Cref{def:good}, which we do in the next three claims.

    \begin{claim}   \label{claim:A}
        $F$ is good.
    \end{claim}

    \begin{proof}[Proof of \Cref{claim:A}]
    We only need to show that $S$ is the only transversal of $F$ of size $k$. Indeed, we can see this by induction on $k$. 
    For $k=1$, suppose $S=\{v\}$ is a transversal of $F$. From \Cref{lem:trans} it follows that $v$ has degree $r$ and is incident to every edge of $F$. Hence $E(F)$ is the set of all edges incident to $v$, and no other vertex in $F$ has degree $r$ since we assume $\ell>r$. Thus, $\{v\}$ is the only transversal for $F$ of size 1.

    Assuming the claim holds for $k$, suppose $\tau(F)=k+1=|E(F)|/r$. 
    %$S=\{v_1,\dots, v_{k+1}\}$ is an independent set in $F$ and  transversal of $F$, with $\deg_F(v_i)=r$ for all $1\le i\le k+1$. 
    Let $u_0$ be a vertex such that $\deg_F(u_0)=r$ and $\deg_F(u_0-1)<r$. Such a vertex must exist, for otherwise we would have $\deg_F(u)=r$ for all $u$, and hence $F=H$, whereas we assumed $F$ is a proper subgraph of $H$. Now for such $u_0$ we must have $\{u_0-r+1,\dots, u_0\}\in E(F)$ and $\{u_0-r,\dots, u_0-1\}\notin E(F)$. Thus, each vertex in $\{u_0-r,\dots, u_0-1\}$ has degree smaller than $r$ and hence cannot lie in a transversal of size $k+1$, by \Cref{lem:trans}. On the other hand, any transversal must intersect the edge $\{u_0-r+1,\dots, u_0\}$, so $u_0$ must lie in any transversal. 
    Thus $F$ includes all $r$ edges incident to $u_0$. For any minimal transversal $T$ for $F$, by \Cref{lem:trans} we must have $|u-u_0|\ge r$ for all $u\in T\setminus\{u_0\}$. 
    Removing $u_0$ and the $r$ incident edges from $F$, we obtain a graph $F'$ with minimal transversal $S'$ of size $k$. By the induction hypothesis, $S'$ is the unique minimal transversal of $F'$. Hence $S=\{u_0\}\cup S'$ is the unique minimal transversal of $F$.  
    \end{proof}

    \begin{claim}   \label{claim:B}
        $F$ satisfies (VG\ref{VG1}).
    \end{claim}

    \begin{proof}[Proof of \Cref{claim:B}]
        Without loss of generality take $v=\ell$. Then $S$ is a subset of $[1,\ell-1]$, with elements $u_1<\cdots<u_k$ satisfying $u_i-u_{i-1}\ge r$ for each $2\le i\le k$.
        For the claim it suffices to produce disjoint edges $e_1,\dots, e_k\subset[1,\ell-1]$ with $u_i\in e_i$ for each $i$. 
        Set $\hat u_k:=\min(u_k, \ell-r)$ and $e_k:=[\hat u_k, \hat u_k+ r-1]$. For $i<k$ we inductively define
        $e_i:=[\hat u_i, \hat u_i+r-1]$ with $\hat u_i := \min(u_i, \hat u_{i+1}-r)$. The separation of elements of $S$ by at least $r$ ensures that the edges are disjoint and contained in $[1,\ell-1]$, with $u_i\in e_i$ for each $i$. The claim follows.
    \end{proof}

    \begin{claim}   \label{claim:C}
        $F$ satisfies (VG\ref{VG2}).
    \end{claim}

    \begin{proof}[Proof of \Cref{claim:C}]
        Without loss of generality take $v=\ell$. Let $S=\{u_1,\dots, u_k\}$ with $u_1=v=\ell$ and $r\le u_2<u_3<\cdots<u_k\le \ell-r$, with $u_i-u_{i-1}\ge r$ for each $i$. 
        Recall that $E(F)$ is the disjoint union of sets $E_1,\dots, E_k$ with $E_i$ the set of $r$ edges incident to $u_i$. Taking $e_0:=[\ell-r+1,\ell]$, $e_1:=\{\ell, 1, 2, \dots, r-1\}$, when $k=1$ we are done.
        For $k\ge2$ our task is to select $e_i\in E_i$ for each $2\le i\le k$ %such that $|e_i\cap e_{i-1}|\le 1$ for all $2\le i\le k$ and $|e_k\cap e_0|\le 1$. 
        so that the collection $\{e_0,e_1,\dots, e_k\}$ has the claimed properties. 

        Let $m\ge1$ and $a\in[r]$ such that $\ell=mr+a$. Then $m=m_\ell$, where
        \[
        m_\ell:=\max\{|T|: T\subset[\ell]\text{ is independent in }H\}.
        \]
        Since we assume $\ell\in\cL_r$, if $m\in[r-1]$ then $\ell\in[(m+1)(r-1)+1,(m+1)r]$. 
        
        First we handle the case that $k<m$. 
        Then $S':=S\setminus\{u_1\}$ is an $r$-separated set of size at most $m-2$ in the interval $[r,\ell-r]$. This interval has length at least $(m-2)r+1$, and hence we can select disjoint edges $e_2,\dots, e_k$ of $F$ such that $u_i\in e_i$ for each $i$, verifying (VG\ref{VG2}). 

        We henceforth assume $k=m$. Thus, $S$ is an independent set for $H$ of maximal size. 

        Next we reduce to the case that $\ell$ is the smallest element $\ell_k\in\cL_r$ with $m_\ell=k$. 
        That is, 
        \[
            \ell_k = \begin{cases}
                (k+1)(r-1)+1 & k\le r-1\\
                kr+1 & k\ge r.
            \end{cases}
        \]
        Indeed, if $\ell>\ell_k$ and $m_\ell=k$, then we can delete $\ell-\ell_r$ vertices from $[r,\ell-r]\setminus S$ to get a new set $S'$ of $k$ elements that is $r$-separated in $C_{\ell_k}^{(r)}$, with an element at $\ell_k$. From an edge covering of $S'$ as in (VG\ref{VG2}), we can add $\ell-\ell_k$ vertices back in to obtain a covering for $S$ with the desired properties. 

        So assume $\ell=\ell_k$. 
        We choose $e_2,\dots, e_k$ as follows:
        \begin{align}
            e_i &:=    \big[(i-1)(r-1)\,,\,i(r-1)\big]\,,  &2\le i\le& \min(k, r-1)\\
            e_{r+j} &:= (r-1)^2 + \big[(k-r)r+1\,,\,(k-r+1)r\big]\,,&0\le j\le &k-r\,.
        \end{align}
        Clearly, $e_0,e_1,\dots, e_{\min(k,r-1)}$ form a loose path, while (if $k\ge r$) $e_r,\dots e_{k-r}$ are disjoint from each other and from $e_0\cup e_1,\dots\cup e_{\min(k,r-1)}$. 
        It only remains to verify that $S\cap e_i\cap e_{i+1}=\emptyset$ for each $1\le i\le \min(k,r-1)-1$. Note we have only assumed $S$ is an independent (i.e. $r$-separated) set in $H$ of size $k$ with $\ell\in S$.

        Note that $e_i\cap e_{i+1}=\{i(r-1)\}$. Suppose toward a contradiction that $i(r-1)\in S$ for some $2\le i\le \min(k,r-1)-1$. Let $S''=S\setminus\{\ell, i(r-1)\}$. We have that $S''$ is an $r$-separated set in $[r,(i-1)(r-1)+1]\cup[(i+1)(r-1)+1,\ell-r]=:I_1\cup I_2$. 
        An $r$-separated subset of an interval $I$ has at most $\lf (|I|-1)/r\rf +1$ elements. Hence,
        \begin{align*}
            k-2 &= |S''| = |S''\cap I_1|+|S''\cap I_2| \\
            &\le \big\lf \tfrac1r\big( (i-1)(r-1)-r+1\big)\big\rf + 
            \big\lf \tfrac1r \big( (k-1)r - (i+1)(r-1) \big)\big\rf + 2\\
            &= \big\lf \tfrac1r(i-2)(r-1)\big\rf + \big\lf \tfrac1r \big( (k-i-2)r + (i+1)\big)\big\rf + 2\\
            &=  i-3 + k-i-2+2\\
            &= k-3
        \end{align*}
        where in the fourth line we used that $i+1\le r-1$ to see that the second of three terms in the third line is $k-i-2$.
        We thus obtain the desired contradiction to conclude the proof.        
    \end{proof}
\end{proof}

\section{The Fano Plane}
\label{sec:Fano}

\subsection{Proof of \Cref{thm:main}(\ref{main.Fano}), upper bound}

We begin by proving the upper bound
\begin{equation}    \label{Fano-UB}
    R_H(n,p,\delta) \le (1+o(1)) \frac1{6}\min\{ \tfrac37\delta, \delta^{\frac 37}\} n^rp^\Delta\log(1/p)
\end{equation}
for $n=\omega(1)$.
From \Cref{prop:upper} it suffices to show
\begin{equation}    \label{Fano:rhoH}
    \rLZ_H(\delta) = \min\{ \tfrac37\delta, \delta^{\frac 37}\}. 
\end{equation}

One easily checks that $H$ has 15 nonzero stable labelings, which are shown in Figure \ref{fig:Fano-labelings}. 
From this we see that
\begin{equation}
    P_H(\xi) = 1+ 7\xi(1) + 7\xi(\tfrac12)^4 + \xi(\tfrac13)^7.
\end{equation}
Furthermore, 
\[
T_H' = \{ (1,0,0), (0,1,0), (0,0,1), (\tfrac12,\tfrac12,0), (\tfrac12,0,\tfrac12), (0,\tfrac12,\tfrac12), (\tfrac13, \tfrac13, \tfrac13)\}
\]
so 
\begin{equation}
    \Vol_H(\xi) = 3 \xi(1) + 3\xi(\tfrac12)^2 + \xi(\tfrac13)^3\,.
\end{equation}
Setting $\al=\xi(\frac13)^3, \beta = \xi(\tfrac12)^2$ and $\gamma=\xi(1)$, we have
\begin{align}   
    \rLZ_H(\delta)
    &= \inf_{\al,\beta, \gam \ge0}\{ \al+ 3\beta + 3\gam: \al^{7/3} + 7\beta^2 + 7\gam \ge\delta\} = \min\{ \tfrac37\delta, \tfrac {3\sqrt7}7\sqrt{\delta}, \delta^{\frac 37}\} = \min\{ \tfrac37\delta, \delta^{\frac 37}\} \label{rhoH-Fano}
\end{align}
where in the second equality we used \Cref{lem:cvx}. Note that for $\delta \in [0, \infty)$, $\frac{3\sqrt 7}{7} \sqrt \delta \le \frac 37 \delta$ if and only if $\delta \ge 7$, and $\frac{3\sqrt 7}{7}\sqrt{\delta} \le \delta^{\frac 37}$ if and only if  $x \le ( \frac{3\sqrt 7}{7})^{-14} < 1$. This means that $\frac{3\sqrt 7}{7} \sqrt \delta$ cannot be the minimizer of $\min\{ \tfrac37\delta, \tfrac {3\sqrt7}7\sqrt{\delta}, \delta^{\frac 37}\}$, which implies the third equality. We thus obtain \eqref{Fano:rhoH} to complete the proof of the upper bound \eqref{Fano-UB}.

\begin{figure}%[th]
    \centering
    \includegraphics[width=\linewidth]{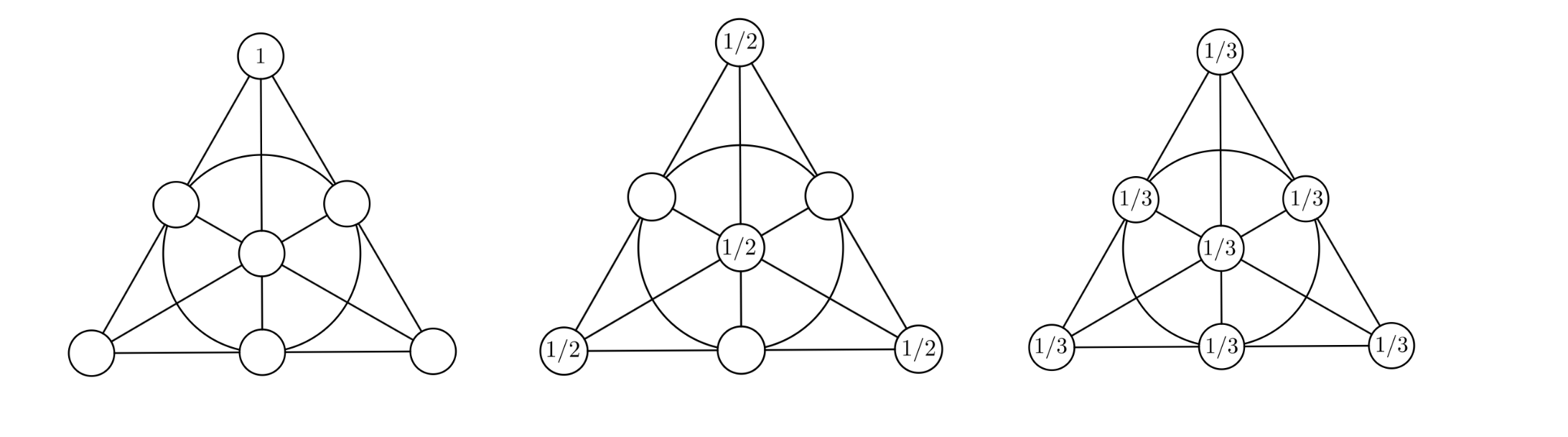}
    \caption{Stable labelings of the Fano plane. Empty vertices are labeled 0. There are $7$ stable labelings of the first type, $7$ stable labelings of the second type, and $1$ stable labeling of the third type. Note that the middle labeling is supported on a proper subgraph, obtained by removing the circular edge, and takes fractional values. This contrasts with the $r$-graphs considered in Sections \ref{sec:gen}--\ref{sec:cycles}, where all labelings supported on proper subgraphs were indicators of independent sets, like the labeling on the left above.}
    \label{fig:Fano-labelings}
\end{figure}

\subsection{Proof of \Cref{thm:main}(\ref{main.Fano}), lower bound}

Now we prove the lower bound
\begin{equation}    \label{Fano-UB}
    R_H(n,p,\delta) \ge (1+o(1)) \frac1{6}\min\{ \tfrac37\delta, \delta^{\frac 37}\} n^rp^\Delta\log(1/p)
\end{equation}
for $n=\omega(1)$.
From \cite[Example 3.8]{CDP} we have
\begin{equation}    \label{Fano:Delta'}
    \Delta'(H)= 3.
\end{equation}
From \Cref{thm:CDP}, \eqref{Fano:Delta'}, \eqref{LB:Phi-phi}, and \Cref{lem:phi-tphi}, to obtain \eqref{Fano-UB} it only remains to establish the following:

\begin{prop}
    \label{prop:Fano}
    Let $(\cX,\mu)$ be a probability space. 
    For any fixed $\delta>0$,
    \begin{equation}    \label{Fano:LB1}
        \frac{\phi_H^{(\cX,\mu)}(p,\delta) }{p^\Delta\log(1/p)} \ge  \min\{ \tfrac37\delta, \delta^{\frac 37}\} + o(1)\,.
    \end{equation}
\end{prop}

\subsection{Proof of \Cref{prop:Fano}}

From \Cref{lem:phi-tphi} it suffices to show that for any fixed $\delta>0$ and symmetric measurable $f$ on $\cX^3$ satisfying \eqref{assu:kappa}, if $t(H,1+f/p)\ge 1+\delta$
then
\begin{equation}
    \label{Fano-goal1}
    p^{-3}J_p(f)\ge \min\{ \tfrac37\delta, \delta^{\frac 37}\} + o(1).
    \end{equation}
Fix such a function $f$ and $\delta>0$. We may assume 
\begin{equation}    \label{assu:K-Fano}
    J_p(f)\lessapprox p^3
\end{equation}
as we are done otherwise. 

\begin{figure}%[th]
    \centering
    \includegraphics[width=0.7\linewidth]{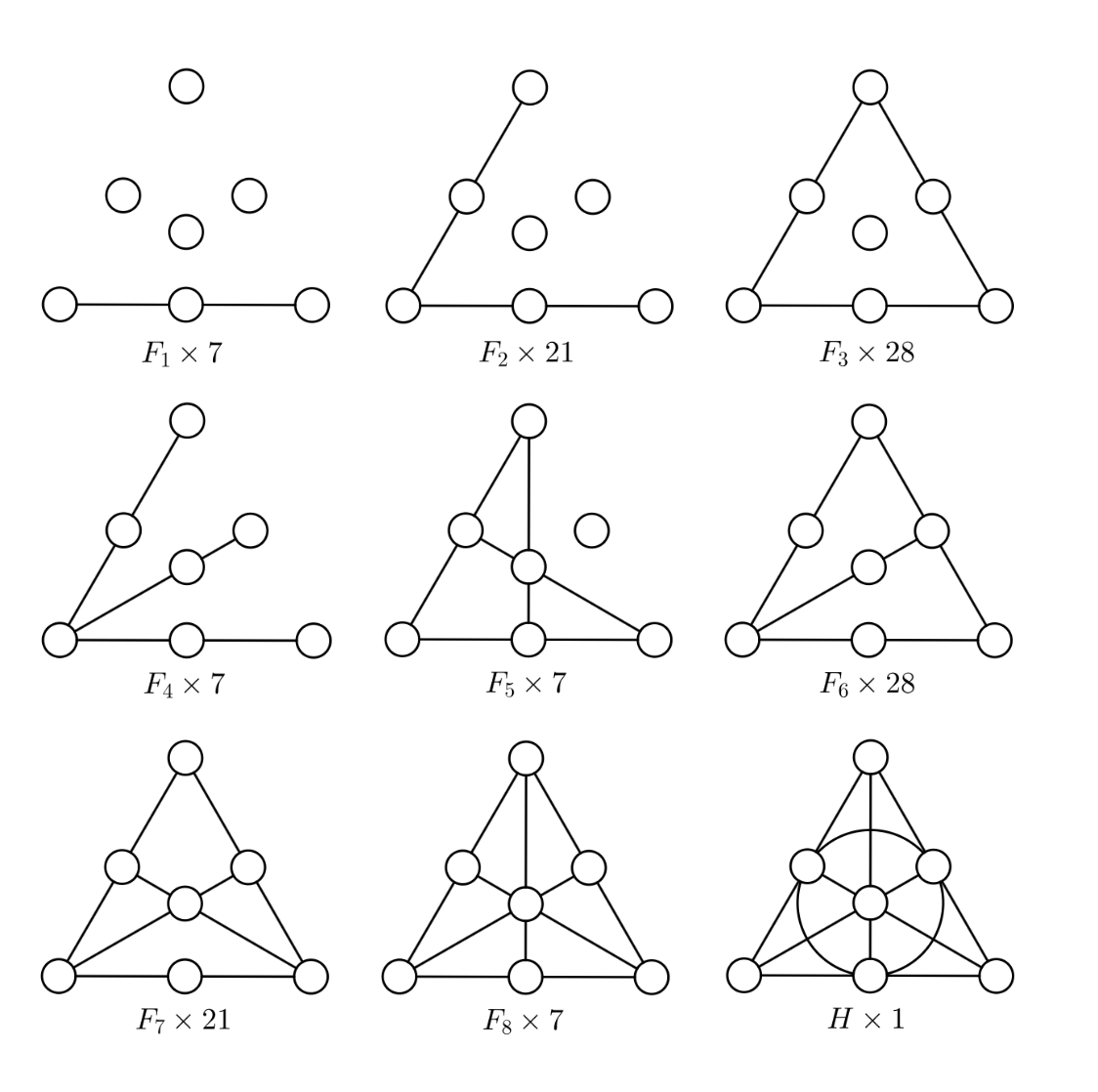}
    \caption{Subgraphs of the Fano plane.} 
    \label{fig:H0subgraphs}
\end{figure}

We have
\begin{align}
 t(H,1+f/p)&= \sum_{F\subseteq H} t(F,f/p)\label{Fano.expand0}\\
    &= 1+ 7t(F_1,f/p) + 21 t(F_2,f/p) + 28t(F_3,f/p) + 7t(F_4,f/p) \notag\\
    &\qquad+ 7t(F_5, f/p) + 28 t(F_6,f/p) + 21 t(F_7,f/p) + 7t(F_8,f/p) + t(H, f/p)\notag
\end{align}
where $F_1,\dots, F_7$ are the nonempty subgraphs of $H$, up to isomorphism, as shown in Figure \ref{fig:H0subgraphs}. 
Since the graphs $F_1,F_2,F_3,F_5$ have maximum degree less than $\Delta(H)=3$, from  \Cref{lem:smallDelta} we get
\begin{equation*}
    7t(F_1,f/p) + 21 t(F_2,f/p) + 28t(F_3,f/p) + 7t(F_5,f/p) =o(1)\,.
\end{equation*}
Note that, $F_6$ and $F_7$ do not have strict stable labelings, because any strict stable labeling of $F_6, F_7$ would yield a valid labeling of $H$, which is as specified in \Cref{fig:Fano-labelings}. However, those stable labelings would correspond to a strict stable labelings of $F_4, F_8$, and $H$ instead, which would be a contradiction.  So from \Cref{lem:noSSL},
\begin{equation*}
    28 t(F_6,f/p) + 21 t(F_7,f/p)=o(1).
\end{equation*}
Thus we have shown
\begin{equation}    \label{Fano.expand1}
    t(H, 1+ f/p) = 1+ 7t(F_4,f/p) + 7f(F_8, f/p) + t(H,f/p) + o(1). 
\end{equation}

Recall the notation \eqref{def:df}--\eqref{def:Db}.
Let $\eps>0$ be a constant to be chosen sufficiently small depending only on $H$, and set
\begin{align}
    D_1 &:= D_{\ge p^\eps}\,,
    \qquad D_{1/2} := D_{\ge p^{3/2+\eps}}\setminus D_{\ge p^\eps}\,,
    \qquad
    D_0 := D_{<p^{3/2+\eps}}\,.
\end{align}

\begin{claim}   \label{claim:Fano.1}
    With the degree-3 vertex of $F_4$ labeled 1, if $\eps>0$ is sufficiently small, then
    (recalling the notation \eqref{def:tH-restricted}),
    \[
    t(F_4, f/p) = t(F_4,f/p \,;\,1\to D_1, \{2,\dots, 7\}\to D_0) + o(1)\,.
    \]
\end{claim}

\begin{claim}   \label{claim:Fano.2}
    With the degree-3 vertices of $F_8$ labeled 1,2,3,4, if $\eps>0$ is sufficiently small, then
    \[
    t(F_8, f/p) = t(F_8,f/p \,;\,\,\{1,2,3,4\}\to D_{1/2}, \{5,6, 7\}\to D_0) + o(1)\,.
    \]
\end{claim}

\begin{claim}
\label{claim:Fano.3}
%    Assume $f$ satisfies \eqref{assu:K-Fano} and \eqref{assu:kappa} with $\Delta=3$. Then for 
    For any fixed $\eps>0$,
    \[
    t(H,f/p) = t(H,f/p; V(H)\to D_{\le p^{2-\eps}})+o(1)\,.
    \]
\end{claim}

Taking these claims as given for now, 
let
\begin{equation}
    \al_f:=p^{-3} \|f1_{D_0^3}\|_{L^3}^3 \,,\quad \beta_f:=p^{-3} \|f1_{D_{1/2}^2\times D_0}\|_{L^3}^3
    \,,\quad
    \gam_f:= p^{-3} \|f1_{D_1\times D_0^2}\|_{L^3}^3\,.
\end{equation}
From \Cref{lem:Finner},
\begin{align*}
       & t(F_4, f/p \,;\,\,1\to D_1, \{2,\dots, 7\}\to D_0) \\
    &\qquad= \int_{\cX^7} \prod_{i=1}^3\frac{f(x,y_i,z_i)}p1_{x\in D_1}1_{y_i,z_i\in D_0} d\mu(x) \prod_{i=1}^3 d\mu(y_i)d\mu(z_i)\\
    &\qquad\le p^{-3} \|f1_{D_1\times D_0^2}\|_{L^3}^3 = \gam_f 
\end{align*}
and
\begin{align*}
       & t(H, f/p \,;\,\,\{1,\dots,4\}\to D_{1/2}, \{5,6, 7\}\to D_0) \\
    &\qquad= \int_{\cX^7} \prod_{e\in H} \frac{f(x_e)}p1_{x_1,\dots, x_4\in D_{1/2}}1_{x_5,x_6,x_7\in D_0}  d\mu^7(x)\\
    &\qquad\le p^{-6} \|f1_{D_{1/2}^2\times D_0}\|_{L^3}^6 = \beta_f^2 \,.
\end{align*}

Note that since $D_{\le p^{2 - \eps}} \subset D_0$ for arbitrarily small $\eps < \tfrac 12$, from Claim \ref{claim:Fano.3}, we have 

\begin{align*}
    &t(H, f/p\,;\,\, V(H) \to D_{\le p^{2 - \eps}}) = t(H, f, V(H) \to D_0) + o(1) \\
    &\qquad = \int_{\cX^7} \prod_{e \in H} \frac{f(x_e)}{p} 1_{x_i \in D_0} \prod_{i=1}^7 \d\mu(x_i) \\
    &\qquad \le p^{-7} \| f1_{D_0^3}\|_{L_3}^7 = \al_f^{\frac73}.
\end{align*}

Combining with \eqref{Fano.expand1} and Claims \ref{claim:Fano.1}, \ref{claim:Fano.2}, \ref{claim:Fano.3},  along with our assumption that $t(H,1+f/p)\ge 1+\delta$, we get that
\[  
    \delta \le \al_f^{\frac 73} + 7\beta^2_f + 7\gam_f + o(1).
\]
On the other hand, from \eqref{Jp-quad},
\begin{align*}
    J_p(f) 
%    \ge  (1+o(1))\int_{\cX^3} f^2 d\mu^3
    \ge  (1+o(1))\int_{\cX^3} f^3 d\mu^{3} 
    \ge (1+o(1))(\al_f + 3\beta_f + 3\gam_f)p^3\,.
\end{align*}
Hence, for some $\delta'=\delta+o(1)$, 
\begin{align*}
    p^{-3}J_p(f) 
    &\ge \inf_{\al,\beta, \gam\ge0}\{ \al_f + 3\beta_f + 3\gam_f : \al_f^{\frac 73} + 7\beta^2_f + 7\gam_f\ge \delta'\} +o(1)
    \\
    &= \min\{ \tfrac37\delta, \tfrac {3\sqrt7}7\sqrt{\delta}, \delta^{\frac 37}\} +o(1) \\
    &= \min\{ \tfrac37\delta,  \delta^{\frac 37}\} +o(1) 
\end{align*}
where for the last equality we recall from \eqref{rhoH-Fano} that the minimum is never attained by $\frac{3\sqrt 7}{7}\sqrt \delta$. 
\qed

\subsection{Proof of \Cref{claim:Fano.1}}

    First we show that that dominant contribution to $t(F_4,f/p)$ comes from embeddings placing the degree-3 vertex 1 in the high-degree set $D_1$, that is:
    \begin{equation}    \label{Fano.1.1}
        t(F_4,f; 1\to \cX\setminus D_1) = o(p^3).
    \end{equation}
    Indeed, the left hand side is 
    \begin{align*}
        \int_{\cX\setminus D_1} d_f(x)^3 d\mu(x) 
        \le p^{2\eps} \int_{\cX}d_f(x)d\mu(x) = p^{2\eps}\|f\|_{L_1}
        \lessapprox p^{3+2\eps} = o(p^3)\,,
    \end{align*}
    where we applied \Cref{lem:L1-bound} in the second inequality.

    Next we reduce to the contribution of embeddings where degree-1 vertices are embedded in the low-degree set $D_0$. We claim
    \begin{equation}    \label{Fano.1.2}
        t(F_4, f \,;\, 1\to D_1, 2\to \cX\setminus D_0) = o(p^3)
    \end{equation}
    if $\eps$ is sufficiently small. 
    Indeed, substituting the pointwise bound $f\le 1$ shows the left hand side is bounded by $\mu(D_1) \mu(\cX\setminus D_0)$, and from Markov's inequality and \Cref{lem:L1-bound},
    \[ 
    \mu(D_1) \mu(\cX\setminus D_0)  
    \le \frac{\|d_f\|_{L_1}}{p^\eps} \frac{\|d_f\|_{L_1}}{p^{3/2+\eps}} = \frac{\|f\|_{L_1}^2}{p^{3/2+2\eps}}
    \lessapprox p^{6-3/2-2\eps} .
    \]
    Taking $\eps<3/4$ yields \eqref{Fano.1.2}.
    From \eqref{Fano.1.1}, \eqref{Fano.1.2} and symmetry it follows that 
    \[
    t(F_4, f) = t(F_4, f \,;\, 1\to D_1, \{ 2, \dots, 7\}\to D_0) + o(p^3)
    \]
    as desired. 
    \qed

\subsection{Proof of \Cref{claim:Fano.2}}
    First we show
    \begin{equation}    \label{Fano.2.1}
        t(H, f \,;\, 1\to D_1) = o(p^6)
    \end{equation}
    if $\eps$ is sufficiently small. 
    Indeed, note that upon removal of the vertex 1 and the three edges incident to it, we obtain the 3-graph $F_3$ (see \Cref{fig:H0subgraphs}), giving the bound
    \begin{align*}
            t(H, f \,;\, 1\to D_1) 
    &\le \int_{\cX^7} 1_{x_1\in D_1} f(x_2, x_3, x_5) f(x_2,x_4, x_6)f(x_3, x_4, x_7) d\mu^7(x)= \mu(D_1) t(F_3, f)\,.
    \end{align*}
    Now from Markov's inequality and \Cref{lem:L1-bound}, 
    \[
    \mu(D_1) \le \frac{\|d_f\|_{L_1}}{p^\eps} \lessapprox p^{3-\eps}
    \]
    while from \Cref{lem:smallDelta} we have
    \[  
    t(F_3,f) \lessapprox p^{9/2}.
    \]
    Taking $\eps<3/2$ gives \eqref{Fano.2.1}.

    Next we bound the contribution of embeddings mapping $1$ to $D_0$. By bounding $f(x_e)\le 1$ for all edges $e\in E(H)$ not containing 1, we get
    \begin{align*}
            t(H, f \,;\, 1\to D_0) 
    &\le t(F_4, f \,;\, 1\to D_0) \\
    &= \int_{D_0} d_f(x)^3d\mu(x)\\
    &\le p^{3+2\eps}\int_\cX d_f(x) d\mu(x) \\
    &= p^{3+2\eps} \|f\|_{L^1}\\
    &\lessapprox p^{6+2\eps} = o(p^6)\,.
    \end{align*}
    From this and \eqref{Fano.2.1} we conclude
    \[
    t(H, f \,;\, 1\to \cX\setminus D_{1/2}) = o(p^6).
    \]
    By symmetry the same holds with the vertex 1 replaced by 2, 3 or 4. We thus have
    \begin{equation}
        t(H, f) = t( H, f \,;\, \{1,2,3,4\}\to D_{1/2}) + o(p^6).
    \end{equation}
    It remains to further reduce to the contribution of embeddings sending $5,6,7$ to $D_0$. By symmetry it suffices to show
    \begin{equation}    \label{Fano.2.3}
        t(H, f \,;\, 5\to \cX\setminus D_0) = o(p^6).
    \end{equation}
    By substituting the pointwise bound $f(x_e)\le 1$ for the two edges $e\in E(H)$ not containing the vertex 5, we get
    \[
    t(H, f \,;\, 5\to \cX\setminus D_0) \le \mu(\cX\setminus D_0) t(F_5, f)
    \]
    with $F_5$ as in \Cref{fig:H0subgraphs}. Since $|E(F_5)|=4$ and $\Delta(F_5)=2$, from \Cref{lem:smallDelta} we get
    $t(F_5,f) \lessapprox p^6$, while Markov's inequality and \Cref{lem:L1-bound} give $\mu(\cX\setminus D_0) \lessapprox p^{3/2-\eps}$. Thus,
    \[
    t(H, f \,;\, 5\to \cX\setminus D_0) \lessapprox p^{7.5-\eps} = o(p^6)
    \]
    for $\eps<3/2$. The claim follows.
    \qed

\subsection{Proof of Claim \ref{claim:Fano.3}}
    Label an arbitrary vertex of $H$ by 1.
    With $b:=p^{2-\eps}$, it suffices to show
    \begin{equation}    \label{Fano.3.1}
        t(H,f\,;\, 1\to D_{>b}) =o(p^7)\,.
    \end{equation}
    Let $H'$ be the subgraph of $H$ obtained by removing vertex 1 and all edges incident to it. 
    Then
    \begin{align*}
        t(H,f\,;\, 1\to D_{> b}) 
        &\le \mu(D_{>b}) t(H', f)\\
        &\lessapprox b^{-1}\|f\|_{L^1} p^{4\cdot 3/2}\\
        &\lessapprox p^{-2+\eps + 3 + 4\cdot 3/2} = p^{7+\eps}
    \end{align*}
    where we applied \Cref{lem:smallDelta} and \eqref{D-Markov}. This gives \eqref{Fano.3.1} to complete the proof.\qed

\bibliographystyle{abbrv} 
\bibliography{rtails}

\begin{thebibliography}{10}

\bibitem{Alon0}
N.~Alon.
\newblock On the number of subgraphs of prescribed type of graphs with a given
  number of edges.
\newblock {\em Israel J. Math.}, 38(1-2):116--130, 1981.

\bibitem{Augeri}
F.~Augeri.
\newblock {Nonlinear large deviation bounds with applications to {W}igner
  matrices and sparse {E}rd\H{o}s--{R}\'{e}nyi graphs}.
\newblock {\em Ann. Probab.}, 48(5):2404--2448, 2020.

\bibitem{BaBa}
A.~Basak and R.~Basu.
\newblock Upper tail large deviations of regular subgraph counts in {E}rd\H
  os-{R}\'enyi graphs in the full localized regime.
\newblock {\em Comm. Pure Appl. Math.}, 76(1):3--72, 2023.

\bibitem{BaKa}
A.~Basak and S.~Karmakar.
\newblock Upper tail bounds for irregular graphs.
\newblock arXiv:2503.05311.

\bibitem{BGLZ}
B.~Bhattacharya, S.~Ganguly, E.~Lubetzky, and Y.~Zhao.
\newblock Upper tails and independence polynomials in random graphs.
\newblock {\em Adv. Math.}, 319:313--347, 2017.

\bibitem{BhDe}
S.~Bhattacharya and A.~Dembo.
\newblock Upper tail for homomorphism counts in constrained sparse random
  graphs.
\newblock {\em Random Structures Algorithms}, 59(3):315--338, 2021.

\bibitem{Chatterjee:survey}
S.~Chatterjee.
\newblock An introduction to large deviations for random graphs.
\newblock {\em Bull. Amer. Math. Soc. (N.S.)}, 53(4):617--642, 2016.

\bibitem{Chatterjee:book}
S.~Chatterjee.
\newblock {\em Large deviations for random graphs}, volume 2197 of {\em Lecture
  Notes in Mathematics}.
\newblock Springer, Cham, 2017.
\newblock Lecture notes from the 45th Probability Summer School held in
  Saint-Flour, June 2015, \'{E}cole d'\'{E}t\'{e} de Probabilit\'{e}s de
  Saint-Flour. [Saint-Flour Probability Summer School].

\bibitem{ChDe}
S.~Chatterjee and A.~Dembo.
\newblock Nonlinear large deviations.
\newblock {\em Adv. Math.}, 299:396--450, 2016.

\bibitem{ChVa}
S.~Chatterjee and S.~Varadhan.
\newblock {The large deviation principle for the {E}rd\H{o}s--{R}\'{e}nyi
  random graph}.
\newblock {\em Eur. J. Comb.}, 32(7):1000--1017, 2011.

\bibitem{Chin:lower}
B.~Chin.
\newblock Structure of {L}ower {T}ails in {S}parse {R}andom {G}raphs.
\newblock {\em Random Structures Algorithms}, 67(1):Paper No. e70028, 2025.

\bibitem{Cohen22}
A.~Cohen~Antonir.
\newblock The upper tail problem for induced 4-cycles in sparse random graphs.
\newblock {\em Random Structures Algorithms}, 64(2):401--459, 2024.

\bibitem{CHMS}
A.~Cohen~Antonir, M.~Harel, F.~Mousset, and W.~Samotij.
\newblock Private communication.

\bibitem{CoDe}
N.~Cook and A.~Dembo.
\newblock {Large deviations of subgraph counts for sparse
  {E}rd\H{o}s-{R}\'{e}nyi graphs}.
\newblock {\em Adv. Math.}, 373:107289, 53, 2020.

\bibitem{CoDe:ergms}
N.~A. Cook and A.~Dembo.
\newblock Typical structure of sparse exponential random graph models.
\newblock {\em Ann. Appl. Probab.}, 34(3):2885--2939, 2024.

\bibitem{CDP}
N.~A. Cook, A.~Dembo, and H.~T. Pham.
\newblock Regularity method and large deviation principles for the {E}rd\H
  os-{R}\'enyi hypergraph.
\newblock {\em Duke Math. J.}, 173(5):873--946, 2024.

\bibitem{Eldan:NMF}
R.~Eldan.
\newblock Gaussian-width gradient complexity, reverse log-sobolev inequalities
  and nonlinear large deviations.
\newblock {\em Geom. Funct. Anal.}, 28(6):1548--1596, 2018.

\bibitem{Finner}
H.~Finner.
\newblock A generalization of holder's inequality and some probability
  inequalities.
\newblock {\em Ann. Probab.}, 20(4):1893--1901, 1992.

\bibitem{FrKa}
E.~Friedgut and J.~Kahn.
\newblock On the number of copies of one hypergraph in another.
\newblock {\em Israel J. Math.}, 105:251--256, 1998.

\bibitem{Gunby}
B.~Gunby.
\newblock Upper tails of subgraph counts in sparse regular graphs.
\newblock Preprint, arXiv:2010.00658.

\bibitem{HMS}
M.~Harel, F.~Mousset, and W.~Samotij.
\newblock Upper tails via high moments and entropic stability.
\newblock {\em Duke Math. J.}, 171(10):2089--2192, 2022.

\bibitem{Janson:LT}
S.~Janson.
\newblock Poisson approximation for large deviations.
\newblock {\em Random Structures Algorithms}, 1(2):221--229, 1990.

\bibitem{JOR}
S.~Janson, K.~Oleszkiewicz, and A.~Ruci{\'{n}}ski.
\newblock Upper tails for subgraph counts in random graphs.
\newblock {\em Israel Journal of Mathematics}, 142(1):61--92, 2004.

\bibitem{JaRu}
S.~Janson and A.~Rucinski.
\newblock The infamous upper tail.
\newblock {\em Random Struct. Algorithms}, 20:317--342, 2002.

\bibitem{JaWa:LT}
S.~Janson and L.~Warnke.
\newblock The lower tail: {P}oisson approximation revisited.
\newblock {\em Random Structures Algorithms}, 48(2):219--246, 2016.

\bibitem{JPPS}
M.~Jenssen, W.~Perkins, A.~Potukuchi, and M.~Simkin.
\newblock Lower tails for triangles inside the critical window.
\newblock arXiv:2411.18563.

\bibitem{KoSa}
G.~Kozma and W.~Samotij.
\newblock Lower tails via relative entropy.
\newblock {\em Ann. Probab.}, 51(2):665--698, 2023.

\bibitem{LiZh}
Y.~P. Liu and Y.~Zhao.
\newblock On the upper tail problem for random hypergraphs.
\newblock {\em Random Struct. Algorithms}, 58(2):179--220, 2021.

\bibitem{LuZh:dense}
E.~Lubetzky and Y.~Zhao.
\newblock On replica symmetry of large deviations in random graphs.
\newblock {\em Random Struct. Algorithms}, 47(1):109--146, 2015.

\bibitem{LuZh:sparse}
E.~Lubetzky and Y.~Zhao.
\newblock On the variational problem for upper tails in sparse random graphs.
\newblock {\em Random Struct. Algorithms}, 50(3):420--436, 2017.

\bibitem{Park:digraphs}
J.~Park.
\newblock Upper tails of subgraph counts in directed random graphs.
\newblock arXiv:2405.01980.

\bibitem{Zhao:lower}
Y.~Zhao.
\newblock On the lower tail variational problem for random graphs.
\newblock {\em Combinatorics, Probability and Computing}, 26(2):301--320, 2017.

\end{thebibliography}

\end{document}